\documentclass[twoside,11pt]{amsart}

\usepackage[english]{babel}
\usepackage[utf8x]{inputenc}
\usepackage{url}
\usepackage{amsmath}
\usepackage{graphicx}
\usepackage[colorinlistoftodos]{todonotes}
\usepackage{amsmath,amscd,amsthm}
\usepackage{amstext, amsfonts, a4}
\usepackage{amssymb}
\usepackage{tikz-cd}
\usepackage{fancyhdr}
\usepackage{enumerate}
\usepackage{cleveref}
\usepackage{xcolor}

\numberwithin{equation}{section}
\newtheorem{thm}[equation]{Theorem}

\newtheorem{prop}[equation]{Proposition}
\newtheorem{cor}[equation]{Corollary}
\newtheorem{lem}[equation]{Lemma}

\theoremstyle{definition}
\newtheorem{exmp}[equation]{Example}

\newtheorem{rem}[equation]{Remark}

\renewcommand{\dim}{\operatorname{\mathsf{dim}}}
\renewcommand{\deg}{\operatorname{\mathsf{deg}}}
\renewcommand{\bmod}{\operatorname{\,\mathsf{mod}}\,}
\newcommand\ind{\operatorname{\mathsf{ind}}}
\renewcommand\exp{\operatorname{\mathsf{exp}}}
\newcommand\op{\operatorname{\mathsf{op}}}
\newcommand\End{\operatorname{\mathsf{End}}}
\newcommand\Br{\operatorname{\mathsf{Br}}}

\newcommand\id{\operatorname{\mathsf{id}}}

\newcommand\Int{\operatorname{\mathsf{Int}}}

\newcommand\Ker{\operatorname{\mathsf{Ker}}}

\newcommand\Nrd{\operatorname{\mathsf{Nrd}}}

\newcommand\corr{\operatorname{\mathsf{cor}}}
\newcommand\hh{\operatorname{\mathbb{H}}}

\newcommand{\car}{\mathsf{char}}
\newcommand{\can}{\operatorname{\mathsf{can}}}

\newcommand{\vf}{\varphi}
\newcommand{\mg}[1]{{#1}^{\times}}
\newcommand{\sq}[1]{{#1}^{\times 2}}
\newcommand{\s}{\sigma}
\newcommand{\nat}{\mathbb{N}}

\newcommand{\la}{\langle}
\newcommand{\ra}{\rangle}

\renewcommand{\leq}{\leqslant}
\renewcommand{\geq}{\geqslant}

\newcommand\Sym{\operatorname{\mathsf{Sym}}}
\newcommand\Skew{\operatorname{\mathsf{Skew}}}
\newcommand\Ad{\operatorname{\mathsf{Ad}}}
\newcommand\ad{\operatorname{\mathsf{ad}}}
\newcommand\sw{\operatorname{\mathsf{sw}}}
\newcommand\rk{\operatorname{\mathsf{rk}}}

\renewcommand{\setminus}{\smallsetminus}
\newcommand{\I}{\mathsf{I}}

\begin{document}
\title{Orthogonal involutions over fields with $\I^3=0$}

\date{1 May, 2025}

\author{Karim Johannes Becher}
\author{Fatma Kader B\.{i}ng\"{o}l}

\address{University of Antwerp, Department of Mathematics, Antwerp, Belgium.}
\email{karimjohannes.becher@uantwerpen.be}

\address{Galatasaray University, Department of Mathematics, Istanbul, Turkey.}
\address{University of Verona, Department of Informatics, Verona, Italy.}
\address{Scuola Normale Superiore, Pisa, Italy.}
\email{fatmakader.bingol@gmail.com}

\begin{abstract}
    We provide upper bounds on the $u$-invariant for skew-hermitian forms over a quaternion algebra with its canonical involution in terms of the $u$-invariant of the base field $F$ of characteristic different from $2$ when $\I^3 F=0$.
    

\medskip\noindent
{\sc Classification} (MSC 2020): 16K20, 16W10, 11E39
\end{abstract}

\maketitle

\section{Introduction}
Let $F$ be a field of characteristic different from $2$. We denote by $\I^3 F$ the third power of the fundamental ideal in the Witt ring of $F$.
R.~Elman and T.Y.~Lam \cite{EL73c} proved that non-degenerate quadratic forms over $F$ are classified by dimension, discriminant and the Clifford invariant precisely when $\I^3 F=0$ (see \Cref{EL:I3=0-classification}).   
Under the same condition that $\I^3F=0$,
E.~Bayer-Fluckiger and R.~Parimala \cite{BFP95} obtained analogous classification results for hermitian forms over $F$-algebras with orthogonal involution. 
From this D.~Lewis and J.-P.~Tignol \cite{LT99} derived classification results for algebras with involution over such fields.
They showed that discriminant and Clifford algebra suffice to classify $F$-algebras with orthogonal involution up to conjugation when $\I^3 F=0$.

Let us say that the field $F$ \emph{satisfies $\I^3=0$} if $F$ has characteristic different from $2$ and  $\I^3F=0$. 
Some of the most classical types of fields where quadratic form theory was first developed have the property $\I^3=0$: $p$-adic fields, non-real number fields and function fields of complex surfaces.
Those fields present a rich theory of quadratic forms and linear algebraic groups.
Often a result was proven first for, say, $p$-adic fields and non-real number fields and later extended to fields with $\I^3=0$.

A remarkable feature concerning quadratic forms over fields which is not controlled by the property $\I^3=0$ is given by the $u$-invariant.
Recall that a quadratic form over $F$ is \emph{isotropic} if it has a non-trivial solution, otherwise it is \emph{anisotropic}.
The \emph{$u$-invariant of $F$}, denoted by $u(F)$, is the supremum in $\nat\cup\{\infty\}$ 
on the dimensions of anisotropic quadratic forms over~$F$.

The classical examples of fields with $\I^3=0$ all have $u$-invariant $1,2$ or $4$.
A.~Merkurjev \cite{Mer91} constructed examples of fields that show that any even positive integer is the $u$-invariant of some field satisfying $\I^3=0$. The value $\infty$ can be realised as well.

In this article, we want to add a new aspect to the study of fields with $\I^3 =0$ by considering the possible degrees of central simple algebras over $F$ carrying an anisotropic orthogonal involution. 
General upper bounds on the degrees in terms of the $u$-invariant of $F$ have been obtained, 
however with no specific attention to fields with $\I^3 =0$. 
In our study we make in particular use of the discriminant and the Clifford algebra of an orthogonal involution on a central simple algebra over $F$. They can be interpreted as cohomological invariants, and in low degrees they can be used to determine whether an involution is isotropic.

We view the involutions on a central simple algebra as the adjoint involutions of hermitian or skew-hermitian forms. 
In \Cref{section:hermit&inv}, we revisit this connection. In \Cref{section:cohomological invariants-inv}, we discuss two classical invariants of orthogonal involutions, namely the discriminant and the Clifford invariant. We further make some remarks related to those invariants for the case of a field with $\I^3 =0$, including a new elementary proof, in this special case, of Merkurjev's theorem  that the $2$-torsion part of the Brauer group is generated by classes of quaternion algebras.
In \Cref{section:orth-hermit-u-inv-quaternion}, we obtain new upper bounds on the orthogonal $u$-invariant of a central simple algebra when $\I^3 =0$.
In \Cref{bound-anis-alg-ind2-orth-I^3=0} we show that, given a field $F$ satisfying $\I^3=0$ 
and a central simple algebra $A$ over $F$ of index $2$ and of degree exceeding $4\lceil \frac{u(F)}4\rceil+6$, every orthogonal involution on $A$ is isotropic. 
In \Cref{orthogonal-hermit-u-quater-I^3=0} we reformulate these results in terms of the orthogonal $u$-invariant of an algebra with involution introduced in \cite{Mah05}.
When $u(F)\geq6$, this bound is better than the bounds that can be obtained  from the results in \cite{Mah05}, \cite{PS13}, \cite{Wu18}; see \Cref{remark:improv-u-inv6-hermit-quat}. 

In \Cref{section:anisotropic-inv-deg8-I^3=0} and \Cref{section:anisotropic-inv-deg10-I^3=0} we  take a closer look on the case where $u(F)=4$.
In this situation, a central simple $F$-algebra admits an orthogonal involution if and only if it is Brauer equivalent to an $F$-quaternion algebra. 
In \Cref{T:exists-quadext-u6}, we construct from a quadratic field extension $K/F$ with $u(K)=6$ an anisotropic $F$-algebra with orthogonal involution of index $2$ and degree $8$ of discriminant given by $K/F$.
In \Cref{deg10-ort-iso-equiv-cond}, we derive a characterisation of the existence of a quadratic field extension $K/F$ with $u(K)=6$ in terms of the presence of an anisotropic $F$-algebra with orthogonal involution of degree $10$.
This allows us to conclude  that our bounds from \Cref{bound-anis-alg-ind2-orth-I^3=0} and \Cref{orthogonal-hermit-u-quater-I^3=0} are optimal when $u(F)=4$; see \Cref{optimality-u^+(Q)-for-u(F)=4}.
We achieve this by constructing algebras with an anisotropic orthogonal involution in degrees $8$ and $10$ over such fields.
Furthermore, the characterisations obtained in \Cref{T:exists-quadext-u6} and \Cref{deg10-ort-iso-equiv-cond} show that the study of questions on isotropy of involutions can be used to study problems relevant to the isotropy of quadratic forms over field extensions which a priori do not involve involutions.

Our main references are \cite{BOI} for the theory of algebras with involution, \cite{Knus91} for the theory of hermitian forms and \cite{EKM08} for the theory of quadratic forms.

We write $\nat$ for the set of natural numbers including zero and $\nat^+=\nat\setminus\{0\}$.

\section{Involutions and hermitian forms}\label{section:hermit&inv} 

Let $A$ be a central simple $F$-algebra. 
The \emph{degree, index and exponent of $A$} are denoted by $\deg A$, $\ind A$ and $\exp A$, respectively. 
Given another central simple $F$-algebra $B$, we write $A\sim B$ to indicate that $A$ and $B$  are Brauer equivalent.
By \cite[Chap.~8, Corollary 1.6]{Schar85},  we have  $A\simeq\mathbb{M}_s(D)$ for a unique $F$-division algebra $D$ and $s=\frac{\deg{A}}{\ind{A}}$, and we say that $A$ is \emph{split} if $D= F$, or equivalently, if $\ind A=1$.
By \cite[Chap.~8, Theorem 1.8]{Schar85}, every finitely generated $A$-right module $V$ decomposes into a direct sum of simple $A$-right modules. 
The number of simple components in a decomposition of $V$ as a direct sum of simple $A$-right modules is called the \emph{rank of} $V$ and denoted by $\rk V$.
It follows that $\rk V=\frac{\dim_FV}{s\cdot\dim_FD}$.
In particular, we have $\rk A=s$, and if $A$ is a division algebra, then $\rk V$ is the dimension of $V$ as an $A$-right vector space.

An \emph{$F$-involution on} $A$ is an anti-auto\-morphism $\sigma: A\to A$ such that $\sigma^2=\id_A$ and $\sigma|_{F}=\id_F$.  
If $A$ admits an $F$-involution, then $A\otimes_F A$ is split.
If $A$ is split, that is
$A\simeq\End_FV$ for some finite-dimensional $F$-vector space $V$,
then by \cite[Chap.~I, Theorem, p.~1]{BOI}, the $F$-involutions on $A$ are in one-to-one correspondence with the non-degenerate symmetric or alternating $F$-bilinear forms on $V$ up to a scalar factor in $F^{\times}$:
The correspondence is given by $b\mapsto\ad_b$ where $\ad_b$ is the \emph{adjoint involution} on $\End_FV$ of a non-degenerate symmetric or alternating $F$-bilinear form $b$ on $V$, which is determined by the rule
$$b(v,f(w))=b(\ad_b(f)(v),w)\quad \text{for all}\, v,w\in V\,\text{and all}\, f\in \End_FV\,.$$	

If $\sigma$ is an $F$-involution on $A$ and $K/F$ is a splitting field of $A$, whereby
$A\otimes_FK\simeq\End_KV$ for some finite-dimensional $K$-vector space $V$, then the $K$-involution $\sigma_K=\sigma\otimes\id_K$ on $A\otimes_FK$ corresponds under such an isomorphism to the adjoint involution of a non-degenerate symmetric or alternating $K$-bilinear form on $V$, and we call $\sigma$ \emph{symplectic} if that form is alternating and \emph{orthogonal} otherwise.
The nature of the associated bilinear form  depends neither on the choice of the splitting field $K/F$ nor on the isomorphism; see \cite[Proposition 2.6]{BOI}.
The property of $\s$ of being orthogonal or symplectic is called the \emph{type of~$\s$}.

Since the adjoint involution determines the bilinear form uniquely up to a factor in $F^{\times}$, orthogonal involutions on the split algebra $\End_FV$ are in one-to-one correspondence with non-degenerate symmetric bilinear forms on $V$ up to scaling. 
Moreover, any non-degenerate symmetric bilinear form on $V$ is given as the polar form $b_\varphi:V\times V\to F,(x,y)\mapsto \frac{1}2(\varphi(x+y)-\varphi(x)-\varphi(y))$
of a unique non-degenerate (regular) quadratic form $\varphi$ defined on $V$,
and we write $\ad_\varphi$ for the corresponding involution $\ad_{b_\varphi}$.
In what follows, we will use this correspondence of orthogonal involutions on $\End_FV$ with non-degenerate quadratic forms up to scaling.

A pair $(A,\s)$ consisting of a central simple $F$-algebra $A$ and an (\emph{orthogonal}, resp.~\emph{symplectic}) $F$-involution $\sigma$ on $A$ will be simply called
an \emph{$F$-algebra with} (\emph{orthogonal}, resp. \emph{symplectic}) \emph{involution}.

Let be given an $F$-algebra with involution $(B,\gamma)$.
Let $V$ be a finitely generated $B$-right module and let $\varepsilon\in\{\pm1\}$.
A bi-additive map $h:V\times V\to B$  is called an $\varepsilon$-\emph{hermitian form over} $(B,\gamma)$ if it satisfies the following:
\begin{itemize}
    \item $h(v\alpha,w\beta)=\gamma(\alpha)h(v,w)\beta$ for all $v,w\in V$ and $\alpha,\beta\in B$,
    \item $h(w,v)=\varepsilon\gamma(h(v,w))$ for all $v,w\in V$.
\end{itemize}
We also simply call $h$ \emph{hermitian} if $\varepsilon=1$ and \emph{skew-hermitian} if $\varepsilon=-1$.
For an $\varepsilon$-hermitian form $h$ over $(B,\gamma)$, we denote by $\rk h$ the rank of the underlying $B$-right module and call it the \emph{rank of $h$}.

A hermitian or skew-hermitian form over $(B,\gamma)$, defined on the $B$-right module $V$, is called \emph{isotropic} if $h(v,v)=0$ holds for some $v\in V\setminus\{0\}$, otherwise it is called \emph{anisotropic}.

Let $h:V\times V\to B$ be an $\varepsilon$-hermitian form over $(B,\gamma)$.
Consider a $B$-right submodule $U$ of $V$. The restriction of $h$ to $U\times U$ is an $\varepsilon$-hermitian form over $(B,\gamma)$, and we denote it by $h|_U$. Such a form $h|_U$ is called a \emph{subform of} $h$.
We define
\begin{equation*}
\begin{split}
    U^{\perp}&=\{v\in V\mid h(v,u)=0\quad\forall u\in U\}\\
    &=\{v\in V\mid h(u,v)=0\quad\forall u\in U\}.
\end{split}
\end{equation*}

The $\varepsilon$-hermitian form $h$ over $(B,\gamma)$ is called \emph{non-degenerate} if $V^{\perp}=\{0\}$, 
that is, if the only vector $v\in V$ such that $h(v,w)=0$ holds for all $w\in V$ is the vector $v=0$. 
Note that any anisotropic $\varepsilon$-hermitian form is non-degenerate.
A submodule $U$ of $V$ is called \emph{totally isotropic} if $h|_U=0$, or equivalently, if $U\subseteq U^{\perp}$.

The $\varepsilon$-hermitian form $h$ is called \emph{hyperbolic} if it is non-degenerate and there exists a totally isotropic $B$-submodule $U$ of $V$ with $\rk U=\frac{1}2 \rk h$, or equivalently, such that $U=U^{\perp}$.

We refer to \cite[Chap. I, \S 2.2 \& \S3.4]{Knus91} for the basic concepts of isometry ($\simeq$) and orthogonal sum ($\perp$)~for $\varepsilon$-hermitian forms.
We set $$\Sym^{\varepsilon}(B,\gamma)=\{\alpha\in B\mid\gamma(\alpha)=\varepsilon\alpha\}\,,$$ which is an $F$-linear subspace of $B$. We also write $\Skew(B,\gamma)=\Sym^{-1}(B,\gamma)$.
We say that the $\varepsilon$-hermitian form $h$ \emph{represents} the element $\alpha\in B$, if there exists $v\in V$ such that $\alpha=h(v,v)$; note that this implies that $\alpha\in\Sym^\varepsilon(B,\gamma)$.

For $\alpha\in\Sym^{\varepsilon}(B,\gamma)$, we denote by $\langle\alpha\rangle$ the $\varepsilon$-hermitian form $B\times B\to B$ given by $(v,w)\to\gamma(v)\alpha w$.
Given $n\in\nat^+$ and $\alpha_1,\ldots,\alpha_n\in\Sym^{\varepsilon}(B,\gamma)$, we denote by $\langle\alpha_1,\ldots,\alpha_n\rangle$ the $\varepsilon$-hermitian form obtained as the orthogonal sum $\langle\alpha_1\rangle\perp\ldots\perp\langle\alpha_n\rangle$ over $(B,\gamma)$.

\begin{prop}\label{half-rk-subform-hyper->iso}
    Assume that $B$ is a division algebra and $h$ is a hyperbolic $\varepsilon$-hermitian form over $(B,\gamma)$. Then every subform $h'$ of $h$ with $\rk h'>\frac{1}{2}\rk h$ is isotropic.
\end{prop}
\begin{proof}
    Let $V$ be the $B$-right vector space on which $h$ is defined.  
    Since $h$ is hyperbolic, there exists a totally isotropic $B$-subspace $W$ of $V$ such that $\dim_BW=\frac{1}{2}\rk h=\frac{1}{2}\dim_BV$.
    Let now $U$ be a $B$-right subspace of $V$ with $\dim_BU>\frac{1}{2}\dim_BV$. In view of the dimensions, we have that $W\cap U\neq\{0\}$. 
    As $h|_{W}=0$, we deduce that $h|_U$ is isotropic. 
\end{proof}

Up to isometry, there is a unique hyperbolic $\varepsilon$-hermitian form of rank $2$, and we denote this form by $\hh$; see \cite[Chap. I, \S3.5]{Knus91}.

\begin{lem}\label{hermit-isot->hyperplane-subform} 
    Assume that $B$ is a division algebra. Let $h$ be a non-degenerate isotropic $\varepsilon$-hermitian form over $(B,\gamma)$. 
    Then $h$ admits a subform isometric to $\hh$. 
    In particular $h$ represents every element of $\Sym^{\varepsilon}(B,\gamma)$.
\end{lem}
\begin{proof}
    See \cite[Corollary 3.7.4]{Knus91} for fact that $h$ admits $\hh$ as a subform. For any $\alpha\in\Sym^{\varepsilon}(B,\gamma)\setminus\{0\}$, we have that the $\varepsilon$-hermitian form $\la \alpha,-\alpha\ra$ over $(B,\gamma)$ is non-degenerate, isotropic and of rank $2$, and consequently isometric to $\hh$, whereby in particular $\alpha$ is represented by $\hh$.
\end{proof}

Let $V$ be a finitely generated $B$-right module. 
Let $h:V\times V\to B$ be a non-degenerate hermitian or skew-hermitian form over $(B,\gamma)$. 
One associates to $h$ an involution $\ad_h$ on $\End_BV$ determined by the rule
$$h(v,f(w))=h(\ad_h(f)(v),w)\quad \text{for all}\, v,w\in V\, \text{and all}\, f\in\End_BV.$$
We call $\ad_h$ the \emph{adjoint involution of} $h$. 

\begin{thm}\label{correspondence-(skew)hermit-involution}
    Any $F$-involution $\sigma$ on $\End_BV$ is the adjoint involution $\ad_h$ of some non-degenerate hermitian or skew-hermitian form $h$ over $(B,\gamma)$, which is unique up to a factor in $F^{\times}$.
    The involutions $\sigma$ and $\gamma$ have the same type if $h$ is hermitian, and they have different types if $h$ is skew-hermitian.
\end{thm}
\begin{proof}
    See \cite[Theorem 4.2]{BOI}.
\end{proof}

We denote by $\Ad_B(h)$ the $F$-algebra with involution $(\End_BV,\ad_h)$.

\begin{exmp}\label{ad-inv-1-dim-hermit}
    For $\varepsilon\in\{\pm1\}$ and $\alpha\in B^{\times}\cap\Sym^{\varepsilon}(B,\gamma)$, we have  $$\Ad_B(\langle\alpha\rangle)\simeq(B,\Int(\alpha^{-1})\circ\gamma).$$
\end{exmp}

Let $(A,\sigma)$ be an $F$-algebra with involution.
We say that $(A,\sigma)$, or simply $\sigma$, is \emph{isotropic} if there exists some $a\in A\setminus\{0\}$ such that $\sigma(a)a=0$, and \emph{anisotropic} otherwise. 
If there exists some $e\in A$ with $e^2=e$ and $\sigma(e)=1-e$, then we call $(A,\sigma)$, or simply $\sigma$, 
\emph{hyperbolic}.
In particular, any hyperbolic involution is isotropic.
Note that, if $A$ is a division algebra, then any involution on $A$ is anisotropic.

The following statement allows us to switch swiftly back and forth between isotropy for a hermitian or skew-hermitian form and for its adjoint involution. We will often use it without explicit reference.

\begin{prop}\label{isotropic-hyperbolic-inv-equiv-cond}
    Let $(A,\sigma)$ be an $F$-algebra with involution.  
    Let $D$ be a central $F$-division algebra Brauer equivalent to $A$ and $\gamma$ an involution on $D$. 
    Let $V$ be a finite-dimensional $D$-right vector space such that $A\simeq\End_DV$ and 
    $h: V\times V\to D$ a non-degenerate $\varepsilon$-hermitian form over $(D,\gamma)$ such that $\sigma=\ad_h$.
    Then $\sigma$ is isotropic (respectively hyperbolic) if and only if $h$ is isotropic (respectively hyperbolic).	
\end{prop}
\begin{proof} 
    See \cite[Corollary 1.8 and Theorem 2.1]{BFST93}.
\end{proof} 

Let $(A,\sigma), (A_1,\sigma_1)$ and $(A_2, \sigma_2)$ be $F$-algebras with involution. 
Following \cite{QMT15}, we say that $(A,\sigma)$ is an \emph{orthogonal sum of $(A_1,\sigma_1)$ and $(A_2,\sigma_2)$}  and we write
$$(A,\sigma)\in(A_1,\sigma_1)\boxplus(A_2,\sigma_2)$$
if there exist an $F$-division algebra with involution $(D,\gamma)$, $\varepsilon\in\{\pm 1\}$ and two non-degenerate $\varepsilon$-hermitian forms  $h_i:V_i\times V_i\to D$ over $(D,\gamma)$ such that $(A_i,\sigma_i)\simeq\Ad_D(h_i)$ for $i=1,2$ and $(A,\sigma)\simeq\Ad_D(h_1\perp h_2)$. 
In particular, we have that $A, A_1$ and $A_2$ are all Brauer equivalent to $D$ and the involutions $\sigma, \sigma_1$ and $\sigma_2$ have the same type. 
Note that, in general, the $F$-algebra with involution $(A,\sigma)$ is not uniquely determined: 
For any $\lambda_1,\lambda_2\in F^{\times}$, the $F$-algebra with involution $\Ad_D(\lambda_1h_1\perp\lambda_2h_2)$ is also an orthogonal sum of $(A_1,\sigma_1)$ and $(A_2,\sigma_2)$, but it is not necessarily isomorphic to $\Ad_D(h_1\perp h_2)$.

\section{Cohomological invariants}\label{section:cohomological invariants-inv}	 
Let $n\in\nat^+$ and let $A$ be a central simple $F$-algebra with $\deg A=2n$. 
Let $\sigma$ be an orthogonal involution on $A$. 
For any $a,b\in A^{\times}\cap\Skew(A,\sigma)$, we have $\Nrd_A(a)\sq{F}=\Nrd_A(b)\sq{F}$, by \cite[Proposition 7.1]{BOI}.
The \emph{discriminant of} $\sigma$ is defined as the class
$$e_1(\sigma)=(-1)^n\Nrd_A(a) F^{\times2} \mbox{ in } F^{\times}/F^{\times2}$$
for an arbitrary $a\in A^{\times}\cap \Skew(A,\s)$.
We say that $\sigma$ has \emph{trivial discriminant} if $e_1(\sigma)$ is the trivial square-class in $F^{\times}/F^{\times2}$.

By an \emph{$F$-quaternion algebra} we mean a central simple $F$-algebra of degree $2$.
Any $F$-quaternion algebra $Q$ carries a unique symplectic involution, called the \emph{canonical involution of $Q$}, and we will denote it by $\can_Q$.  
For $a,b\in\mg{F}$, we denote by $(a,b)_F$ the $F$-quaternion algebra generated by two elements $i$ and $j$ satisfying the relations  $i^2=a,j^2=b\mbox{ and }ij+ji=0$.
Any $F$-quaternion algebra is isomorphic to $(a,b)_F$ for certain $a,b\in\mg{F}$.

For later use we make the following observation.
\begin{lem}\label{1-dim skew-herm-over-quaternion}
    Let $Q$ be an $F$-quaternion algebra. Let $h$ be a non-degenerate skew-hermitian form over $(Q,\can_Q)$. 
    Assume that $h$ represents an element $\alpha\in Q^{\times}$ with $-\Nrd_Q(\alpha)\in F^{\times2}$. 
    Then $Q$ is split and $\Ad_Q(h)$ is isotropic. 
\end{lem}
\begin{proof}
    Since $\alpha$ is represented by $h$, we deduce that $\alpha$ is a pure quaternion in $Q$. 
    Hence $-\Nrd_Q(\alpha)=-\alpha\can_Q(\alpha)=\alpha^2$. 
    It follows by our assumption that there exists some $c\in F^{\times}$ such that $\alpha^2=c^2$. 
    Hence $(\alpha-c)(\alpha+c)=0$ in $Q$. Since $\alpha\not\in F$, this implies that $Q$ is split.
	
    Let now $V$ be the underlying $Q$-right module of $h$ and let $u\in V$ be such that $h(u,u)=\alpha$. 
    Set $v=u(\alpha+c)$. 
    Note that $\alpha\in\mg{Q}\setminus F$ and therefore $h(u,v)=\alpha(\alpha+c)\neq 0$.
    In particular $v\in V\setminus\{0\}$.
    As $$h(v,v)=h(u(\alpha+c),v)=\can_Q(\alpha+c)h(u,v)=(c-\alpha)\alpha(\alpha+c)=\alpha(c^2-\alpha^2)=0\,,$$ 
    we obtain that $h$ is isotropic.
    Hence $\ad_h$ is isotropic, by \Cref{isotropic-hyperbolic-inv-equiv-cond}.
\end{proof}

\begin{prop}\label{L2}
    Let $(A_0,\sigma_0)$ and $(A_1,\sigma_1)$ be $F$-algebras with orthogonal involution of index $2$ and such that $A_0\sim A_1$.
    Assume that $(A_0,\sigma_0)\boxplus (A_1,\s_1)$ contains some isotropic $F$-algebra with involution.
    Then there exists a quadratic field extension $L/F$ such that $(A_i,\sigma_i)_L$ is split and isotropic for $i=0,1$.
\end{prop}
\begin{proof}
    Let $Q$ denote the $F$-quaternion division algebra which is Brauer equivalent to $A_0$ and $A_1$.
    By the assumption, there exist two non-degenerate skew-hermitian forms $h_0$ and $h_1$  
    over $(Q,\can_Q)$ with $(A_i,\sigma_i)\simeq\Ad_Q(h_i)$ for $i=0,1$ and such that $h_0\perp h_1$ is isotropic. 
  
    We claim that there exists an element $\alpha\in \Skew(Q,\can_Q)\setminus\{0\}$ which is represented by $h_0$ as well as by $-h_1$. 
    The fact that $h_0\perp h_1$ is isotropic readily implies this if one assumes $h_0$ and $h_1$ to be both anisotropic.
    On the other hand,  by \Cref{hermit-isot->hyperplane-subform} every isotropic non-degenerate skew-hermitian form over $(Q,\can_Q)$ represents every element of $\Skew(Q,\can_Q)$. 
    Therefore the claim also follows if any of the forms $h_0$ and $h_1$ is isotropic.
 
    We fix such an element $\alpha$ and set $L=F(\alpha)$.
    Since $Q$  is a division algebra and $\alpha\in Q\setminus F$, we have that $L$ is a quadratic field extension of $F$ contained in $Q$. 
    Hence $Q_L$ is split. 
    As $-\Nrd_Q(\alpha)=\alpha^2\in \sq{L}$ and $\alpha$ is represented by $h_0$ and $-h_1$, we conclude by \Cref{1-dim skew-herm-over-quaternion} that $(A_i,\sigma_i)_L$ is split and isotropic for $i=0,1$. 
 \end{proof}

For an $F$-algebra with orthogonal involution $(A,\sigma)$, we denote by $C(A,\sigma)$ the \emph{Clifford algebra of $(A,\sigma)$}. 
It is a semi-simple $F$-algebra; we refer to \cite[Section 8]{BOI} for its definition.

\begin{exmp}\label{clifford-ort-inv-biquat}
    Let $Q_1,Q_2$ be $F$-quaternion algebras and set $A=Q_1\otimes_FQ_2$. 
    Let further $\sigma=\can_{Q_1}\otimes\can_{Q_2}$. Then $\sigma$ is an orthogonal involution on $A$. 
    It is shown in \cite[Example 8.19]{BOI} that $C(A,\sigma)\simeq Q_1\times Q_2$. 
\end{exmp}

For a quadratic field extension $K/F$ and a central simple $K$-algebra $A$, we denote by $\corr_{K/F}A$ the corestriction of $A$ to $F$ as defined in \cite[Section 8]{Draxl}.
The following theorem determines the structure of the Clifford algebra of an algebra with orthogonal involution of even degree.

We denote by $\Br(F)$ the Brauer group of $F$. We sometimes write $[A]$ for the class in $\Br(F)$ given by a central simple $F$-algebra $A$.

\begin{thm}\label{structure-clifford-awi}
    Let $n\in\nat^+$ and let $(A,\sigma)$ be an $F$-algebra with orthogonal involution  with $\deg A=2n$.
    Let $d\in \mg{F}$ be such that $e_1(\s)=d \sq{F}$.
    Let $K$ be the center of $C(A,\s)$. 
    Then $K$ is a quadratic  étale $F$-algebra isomorphic to $F[X]/(X^2-d)$, and we have the following two cases:
\begin{enumerate}[$(1)$]
    \item If $e_1(\s)$ is non-trivial, then $K$ is a field, $C(A,\sigma)$ is a central simple $K$-algebra of degree $2^{n-1}$,  
    $2[C(A,\sigma)]=n[A_K]$ in $\Br(K)$ and $[\corr_{K/F} C(A,\s)]=(n+1)[A]$ in $\Br(F)$.
    \item If $e_1(\s)$ is trivial, then $K\simeq F\times F$ and $C(A,\sigma)\simeq C^+\times C^-$ for two central simple $F$-algebras $C^+$ and $C^-$ of degree $2^{n-1}$ satisfying the relations $2[C^+]=2[C^-]=n[A]$ and $[C^+]+[C^-]=(n+1)[A]$ in $\Br(F)$.
\end{enumerate}	 
\end{thm}
\begin{proof}
    See \cite[Theorem 8.10, Theorem 9.12]{BOI}. 
\end{proof}

Let $(A,\sigma)$ be an $F$-algebra of even degree with orthogonal involution and assume that $e_1(\sigma)$ is trivial.
By \Cref{structure-clifford-awi}, we have that $C(A,\sigma)\simeq C^+\times C^-$ for two central simple $F$-algebras $C^+$ and $C^-$ with $[C^-]=[C^+]+[A]$ in $\Br(F)$. 
Note that, since $\exp A\leq 2$, the set $\{0,[A]\}$ is a subgroup of $\Br(F)$, which we denote by $\langle[A]\rangle$. 
The \emph{Clifford invariant of} $(A,\sigma)$, denoted by $e_2(\sigma)$, is defined as the class given by $[C^+]$ and $[C^-]$ in the quotient group $\Br(F)/\langle [A]\rangle$.

If $A\simeq\End_FV$ for some finite-dimensional $F$-vector space $V$ and $\sigma=\ad_\varphi$ for some non-degenerate quadratic form $\varphi$ on $V$, then $[A]=0$, $e_1(\sigma)$ is the discriminant of $\varphi$ and $e_2(\sigma)=[C(\varphi)]$ where $C(\varphi)$ is the Clifford algebra of $\varphi$;
see \cite[Proposition 7.3, Proposition 8.8]{BOI}.

By \cite[Proposition 7.3]{BOI} and \cite[Proposition 8.31]{BOI}, respectively, the invariants $e_1$ and $e_2$ are trivial for hyperbolic involutions. 
Furthermore, they commute with scalar extension.

\begin{prop}\label{iso=hyper-for-deg4+orth+trivialdisc}
    Let $Q$ and $Q'$ be $F$-quaternion algebras and let $$(A,\sigma)=(Q,\can_Q)\otimes(Q',\can_{Q'}).$$
    Then $\sigma$ is isotropic if and only if $Q$ or $Q'$ is split.
\end{prop}
\begin{proof}
    By \cite[Proposition 2.10]{BFPQM}, $\s$ is isotropic if and only if it is hyperbolic.
    By \cite[Proposition 4.10]{BU18}, $\sigma$ is hyperbolic if and only if $Q$ or $Q'$ is split. 
\end{proof}

\begin{cor}\label{quad-ext-hyper->trivial-disc}
    Let $(A,\sigma)$ be an $F$-algebra with orthogonal involution such that $\deg A=4$. 
    Let $K/F$ be a quadratic field extension. Assume that $(A,\sigma)_K$ is hyperbolic. 
    Then $e_1(\sigma)$ is trivial or $A$ is split.
\end{cor}
\begin{proof}
    Let $a\in\mg{F}$ be such that $K\simeq_F F(\sqrt{a})$.
    Since $\sigma_K$ is hyperbolic, $e_1(\sigma_K)$ is trivial.
    It follows that either $e_1(\sigma)=\sq{F}$ or $e_1(\sigma)=a\sq{F}$ in $F^{\times}/F^{\times 2}$.
    Suppose now that $e_1(\sigma)=a\sq{F}$.  
    It follows by \cite[Theorem 15.7]{BOI} that there exists a $K$-quaternion algebra $Q$ such that $C(A,\sigma)\simeq Q$, $\corr_{K/F}Q\simeq A$ and 
    $(A,\sigma)_K\simeq(Q,\can_{Q})\otimes_K (^{\iota}Q,\can_{^{\iota}Q})$ where $\iota$ is the non-trivial $F$-automorphism of $K$ and ${}^{\iota}Q$ denotes the $K$-quaternion algebra obtained by letting $K$ act on $Q$ via $\iota$.
    Since $\sigma_K$ is hyperbolic, by \Cref{iso=hyper-for-deg4+orth+trivialdisc}, either $Q$ or $^{\iota}Q$ is split.
    As $\corr_{K/F}Q\simeq \corr_{K/F} (^{\iota}Q)\simeq A$, we conclude that $A$ is split. 
\end{proof}

\begin{prop}\label{ort-sum-rel-e_1-e_2}
    Let $(A_1,\sigma_1)$, $(A_2,\sigma_2)$ and $(A,\s)$ be $F$-algebras with orthogonal involution such that $(A,\sigma)\in(A_1,\sigma_1)\boxplus(A_2,\sigma_2)$. 
    Then the following hold:
\begin{enumerate}[$(a)$]
    \item $\deg A=\deg A_1+\deg A_2$.
    \item If $\deg A_1\equiv\deg A_2\equiv0\bmod 2$, then $e_1(\sigma)=e_1(\sigma_1)\cdot e_1(\sigma_2)$ in $F^{\times}/F^{\times2}$.
    \item If $\deg A_1\equiv\deg A_2\equiv0\bmod 2$ and both $e_1(\sigma_1)$ and $e_1(\sigma_2)$ are trivial, then $e_2(\sigma)=e_2(\sigma_1)+e_2(\sigma_2)$ in $\Br(F)/\langle[A]\rangle$.
\end{enumerate}
\end{prop}
\begin{proof}
    See \cite[Proposition 3.1]{QMT15}.
\end{proof}

\begin{lem}\label{comput-disc-orth-product-of-quaternions}
    Let $Q$ be an $F$-quaternion division algebra. 
    Let $n\in\nat^+$ and $\alpha_1,\ldots,\alpha_n\in\Skew(Q,\can_Q)\setminus\{0\}$. 
    Over $(Q,\can_Q)$, consider the skew-hermitian form $h=\langle\alpha_1,\ldots,\alpha_n\rangle$.
    Then $\Ad_Q(h)$ is an $F$-algebra with orthogonal involution such that $e_1(\ad_{h})=(-1)^n\Nrd_Q(\alpha_1\cdots\alpha_n)\sq{F}$.
\end{lem}
\begin{proof}
    Set  $\sigma=\Int(\alpha_1^{-1})\circ\can_Q$. 
    Then $\sigma$ is an orthogonal involution on $Q$ and $\Ad_Q(\langle\alpha_1\rangle)\simeq(Q,\sigma)$.
    Since $\sigma(\alpha_1)=-\alpha_1$, we have $e_1(\sigma)= -\Nrd_Q(\alpha_1)\sq{F}$.
    This shows the statement for $n=1$.
    As $\Ad_Q(h)\in\Ad_Q(\langle\alpha_1\rangle)\boxplus\ldots\boxplus\Ad_Q(\langle\alpha_n\rangle)$, the general case follows by induction on $n$ using \Cref{ort-sum-rel-e_1-e_2}~$(b)$.
\end{proof}

\begin{lem}\label{exists-orthogonal-given-disc-I^3=0} 
    Let $Q$ be an $F$-quaternion division algebra. Let $d=\Nrd_Q(x)$ for some $x\in\mg{Q}$.
    For any $n\in\nat $ with $n\geq2$, there exists a non-degenerate skew-hermitian form $h$ over $(Q,\can_Q)$ of rank $n$ such that $e_1(\ad_{h})=(-1)^nd F^{\times2}$.
\end{lem}
\begin{proof}	
    Write $Q=F\oplus Q'$ where $Q'=\Skew(Q,\can_Q)$. As $Q'x$ is a $3$-dimensional $F$-subspace of $Q$, we have that $Q'\cap Q'x\neq\{0\}$. 
    Hence, there exist $\alpha,\beta\in Q'\setminus\{0\}$ such that $\beta=\alpha x$, whence $\alpha^{-1}\beta=x$.
    Along the same lines, using that $n\geq2$, we can find $\alpha_1,\ldots,\alpha_n\in Q'$ such that $x=\alpha_1\cdots\alpha_n$.
    Set $h=\langle\alpha_1,\ldots,\alpha_n\rangle$. 
    Then $h$ is a non-degenerate skew-hermitian form over $(Q,\can_Q)$ of rank $n$, and by \Cref{comput-disc-orth-product-of-quaternions}, we have $e_1(\ad_{h})=(-1)^n\Nrd_Q(\alpha_1\cdots\alpha_n) F^{\times2}=(-1)^nd F^{\times2}$.
    This concludes the proof.
\end{proof}

In small degrees, the discriminant is useful for determining whether an algebra with involution is decomposable, or isotropic.
\begin{prop}\label{app-disc-deg2-4-6}
    Let $(A,\sigma)$ be an $F$-algebra with orthogonal involution. 
\begin{enumerate}[$(a)$]
    \item If $\deg A=4$, then $e_1(\sigma)$ is trivial if and only if there exist $F$-quaternion algebras $Q$ and $Q'$ such that $(A,\sigma)\simeq(Q,\can_Q)\otimes(Q',\can_{Q'})$.
    \item If $\deg A=6$, $A$ is not split and $e_1(\sigma)$ is trivial, then $(A,\sigma)$ is anisotropic.
\end{enumerate}	
\end{prop}
\begin{proof}
    $(a)$ See \cite[Theorem 3.1]{KPS91} and \cite[Corollary 15.12]{BOI} for two different proofs of this fact.

    $(b)$ Assume that $\deg A=6$ and $A$ is not split. Since $A$ carries an orthogonal involution, we have $\exp A\leq2$. 
    Since $\ind A$ divides $\deg A$ and $A$ is not split, we obtain that $\ind A=2$. Hence $A\sim Q$ for some $F$-quaternion division algebra $Q$. 
    Thus $(A,\sigma)\simeq\Ad_Q(h)$ for a non-degenerate skew-hermitian form $h$ over $(Q,\can_Q)$ with $\rk h=3$. 
    Assume that $(A,\sigma)$ is isotropic. Then $h\simeq\hh\perp\langle \alpha\rangle$ for a pure quaternion $\alpha\in Q^{\times}$. 
    As $Q$ is not split, it follows by \Cref{1-dim skew-herm-over-quaternion} that $-\Nrd(\alpha)\notin F^{\times2}$.
    Using \Cref{ort-sum-rel-e_1-e_2}~$(b)$ and \Cref{comput-disc-orth-product-of-quaternions} we compute that $e_1(\sigma)=e_1(\ad_{\langle\alpha\rangle})=-\Nrd(\alpha) F^{\times2}$. 
    Thus $e_1(\sigma)$ is non-trivial.
\end{proof}   

We now revisit some known facts on quadratic forms and algebras with involution over fields with $\I^3=0$. The following result was obtained by Elman and Lam in \cite[Theorem 3.11]{EL73c}.

\begin{prop}\label{EL:I3=0-classification}
    Non-degenerate quadratic forms over $F$ are classified up to isometry by their dimension, discriminant and Clifford invariant if and only if $\I^3F=0$.
    In particular, if $\I^3 F=0$, then a non-degenerate quadratic form over $F$ is hyperbolic if and only if it is even-dimensional, has trivial discriminant and trivial Clifford invariant.
\end{prop}
\begin{proof}
    See e.g. \cite[Theorem 35.10]{EKM08}.
\end{proof}

We denote by $\Br_2(F)$ the $2$-torsion subgroup of the Brauer group $\Br(F)$. Note that $[Q]\in\Br_2(F)$ for any $F$-quaternion algebra $Q$.
Merkurjev showed in \cite{Mer81} that the Clifford invariant induces a natural isomorphism 
$$\I^2 F/\I^3 F\to \Br_2(F)\,.$$ The injectivity of this homomorphism can be seen as an extension of \Cref{EL:I3=0-classification}.
The surjectivity on $\Br_2(F)$ essentially says that $\Br_2(F)$ is generated by the classes of $F$-quaternion algebras.
This very deep theorem turns out to have a far more elementary proof for fields with $\I^3=0$.
Since we will use it in this context, we take the occasion of pointing this out. 
It derives in a relatively straightforward way from \Cref{EL:I3=0-classification}, however using the fact that any element of $\Br_2(F)$ is split by some $2$-extension.
An elementary proof of this latter fact -- independent of Merkurjev's Theorem, from which it follows immediately -- was given in \cite{Bec16}.

\begin{lem}\label{L:I30-Merk}
    Assume that $\I^3F=0$. 
    Let $K/F$ be a quadratic field extension and let $\alpha\in\Br_2(F)$.
    Assume that $m\in\nat$ is such that $\alpha_{K}$ is the class of a tensor product of $m$\, $K$-quaternion algebras.
    Then $\alpha$ is the class of a tensor product of $m+1$\, $F$-quaternion algebras.
\end{lem}
\begin{proof}
    It follows from the hypothesis that $\alpha_K=[C(\varphi)]$ for a non-degenerate quadratic form $\varphi$ over $K$ with $\dim\varphi=2m+2$ and with trivial discriminant; see \cite[Lemma 5]{Mer91}.
    We choose an $F$-linear map $s:K\to F$  with $\Ker s= F$ and
    consider the transfer $s_\ast\varphi$ of the quadratic form $\varphi$ with respect to $s$.
    By \cite[Lemma 3.1]{BGBT18}, we have that $0=2\alpha= \corr_{K/F}(\alpha_F)=[C(s_\ast\varphi)]=e_2(s_\ast\varphi)$ in $\Br_2(F)$.
    Since $\I^3 F=0$, we get by \Cref{EL:I3=0-classification} that $s_\ast\varphi$ is hyperbolic.
    By \cite[Theorem 34.9]{EKM08}, we obtain that $\varphi\simeq\psi_{K}$ for some non-degenerate quadratic form $\psi$ over $F$ with trivial discriminant.
    Set $\beta=[C(\psi)]$ in $\Br_2(F)$. Note that $\dim\psi=\dim\varphi=2m+2$.
    It follows that $\beta$ is the class of a tensor product of $m$\, $F$-quaternion algebras.
    Furthermore, $C(\psi_K)\simeq C(\varphi)$ and hence $\beta_K=[C(\varphi)]=\alpha_K$.
    Therefore $(\alpha-\beta)_K=0$.
    As $[K:F]=2$, it follows that $\alpha-\beta=[Q]$ for an $F$-quaternion algebra $Q$.
    We conclude that $\alpha=\beta+[Q]$, which is the class of a tensor product of $m+1$\, $F$-quaternion algebras.
\end{proof}

\begin{prop}\label{P:I30-Merk}
    Assume that $\I^3F=0$. For $\alpha\in\Br_2(F)$, there exists a unique anisotropic even-dimensional quadratic form $\varphi$ of trivial discriminant over $F$ such that $\alpha=[C(\varphi)]$. 
    In particular, $\Br_2(F)$ is generated by the classes of $F$-quaternion algebras. 
\end{prop}
\begin{proof}
    Let $\alpha\in\Br_2(F)$.
    By \cite[Theorem]{Bec16}, there exist $r\in\nat$ and a sequence of field extensions $(F_i)_{i=0}^r$ of $F$ with $F_0=F$, $\alpha_{F_r}=0$ and $[F_i:F_{i-1}]=2$ for $1\leq i\leq r$.
    As $\I^3 F=0$, we obtain by induction on $i$ using \cite[Corollary~35.8]{EKM08} that $\I^3 F_i=0$ for $1\leq i\leq r$.
    For $0\leq j\leq r$, we conclude by induction on $j$ using \Cref{L:I30-Merk} that $\alpha_{F_{r-j}}$ is the class of a tensor product of $j$\, $F_{r-j}$-quaternion algebras. 

    Hence $\alpha$ is the class of a tensor product of $r$\, $F$-quaternion algebras.
    This implies that $\alpha=[C(\psi)]$ for a non-degenerate quadratic form of trivial discriminant over $F$ with $\dim\psi=2r+2$.
    Since $\I^3 F=0$, it follows by \Cref{EL:I3=0-classification} that $\psi$ is uniquely determined up to Witt equivalence by $[C(\psi)]$.
    Let $\varphi$ be the anisotropic part of $\psi$. 
    Then $\varphi$ has trivial discriminant, too. We conclude that $\varphi$ is the unique even-dimensional anisotropic quadratic form of trivial discriminant over $F$ with $[C(\varphi)]=[C(\psi)]=\alpha$.
\end{proof}

\Cref{EL:I3=0-classification} is extended from quadratic forms to orthogonal involutions in \cite[Theorem A]{LT99}.
There it is shown that, if $\I^3F=0$, then orthogonal involutions on a central simple $F$-algebra are classified by their discriminant and Clifford invariant. 
(The proof involves the fact formulated in \Cref{P:I30-Merk}.)
The following consequence and variation is obtained in \cite{BQM23}.

\begin{prop}\label{deg4m-ortogonal-isot-hyper-I^3F=0}
    Assume that $\I^3F=0$. Let $(A,\sigma)$ be an $F$-algebra with orthogonal involution of trivial discriminant and with  $\deg A\equiv 0\bmod 4$. 
    Then the following hold: 
\begin{enumerate}[$(a)$]
    \item If $\deg A>4$ and both simple components of $C(A,\s)$ have index at most $2$, then $\sigma$ is isotropic.
    \item If any of the simple components of $C(A,\s)$ is split, then $\sigma$ is hyperbolic.
\end{enumerate}
\end{prop}
\begin{proof}
    See \cite[Proposition 4.2]{BQM23}.
\end{proof}

\section{Isotropy over a quadratic field extension} 

Consider a quadratic field extension $K/F$.
For a non-degenerate quadratic form $\varphi$ defined  over $F$,  the anisotropic part of the quadratic form $\varphi_K$ obtained by extending scalars to $K$ is itself defined over $F$ (see \cite[Example 29.2]{EKM08}).
In this section, we show that the same phenomenon appears for non-degenerate skew-hermitian forms over a quaternion division algebra with its canonical involution.
This will be presented by a distinction of two cases, according to whether the given quaternion algebra remains a division algebra or splits by scalar extension to $K$.

Let $(B,\gamma)$ be an $F$-algebra with involution, $\varepsilon\in\{\pm1\}$ 
and $h:V\times V\to B$ an $\varepsilon$-hermitian form over $(B,\gamma)$, where $V$ is a finitely generated $B$-right module.
Consider a field extension $K/F$.
By scalar extension from $F$ to $K$, we obtain an $\varepsilon$-hermitian form $h_K: V\otimes_FK\times V\otimes_FK\to B\otimes_FK$ over $(B\otimes_FK,\gamma_K)$ given by $h_K(v\otimes\alpha,w\otimes\beta)=h(v,w)\otimes\alpha\beta$ for  $v,w\in V,\alpha,\beta\in K$.

The following fact is well-known, see e.g. the proof of \cite[Chapter 10, Theorem 3.2]{Schar85}. For convenience we include a proof.
\begin{prop}\label{quad-ext-excellent-skewhermit/Q-split} 
    Let $Q$ be an $F$-quaternion division algebra and let $K/F$ be a quadratic field extension such that
    $Q_K$ is split. 
    Fix $\alpha\in Q\setminus F$ with $\alpha^2\in F^{\times}$ and $K\simeq_F F(\alpha)$.
    Let $h$ be a non-degenerate skew-hermitian form over $(Q,\can_Q)$. 
    There exist $n\in\nat $, $\lambda_1,\ldots,\lambda_n\in F^{\times}$ and a non-degenerate skew-hermitian form $h'$ over $(Q,\can_Q)$ such that $h'_K$ is anisotropic and 
    $$h\simeq\langle\alpha\lambda_1,\ldots,\alpha\lambda_n\rangle\perp h'\,.$$
    Moreover, given any such decomposition, the form $\langle\alpha\lambda_1,\ldots,\alpha\lambda_n\rangle_K$ is hyperbolic and $e_1(\ad_{\langle\alpha\lambda_1,\ldots,\alpha\lambda_n\rangle})=(-\Nrd_Q(\alpha))^n\sq{F}$.
\end{prop}
\begin{proof}
    If $h_K$ is anisotropic, then we may simply take $n=0$ and $h'=h$. Assume therefore now that $h_K$ is isotropic.
    We  identify $K$ with the subfield $F(\alpha)$ of $Q$.
    It follows by the Skolem-Noether Theorem that $Q=K\oplus\beta K$ for some $\beta\in Q\setminus F$ with $\beta^2\in F^{\times}$ and $\beta\alpha=-\alpha\beta$.
    Let $V$ denote the $Q$-right vector space on which $h$ is defined. 
    Consider the two projections $\pi_0,\pi_1:Q\to K$ determined by the decomposition $Q=K\oplus\beta K$, that is such that $z=\pi_0(z)+\beta\pi_1(z)$ for any $z\in Q$.
    Set $h_0=\pi_0\circ h$ and $h_1=\pi_1\circ h$.    
    Then $h_0$ is a non-degenerate skew-hermitian form over $(K,\iota)$, where $\iota$ is the non-trivial $F$-autormorphism of $K$, and $h_1$ is a non-degenerate symmetric bilinear form over $K$. 
	
    Any $\lambda\in K$ induces a $K$-linear map $r_{\lambda}:V\to V, v\mapsto v\lambda$.
    Now the map $\End_QV\times K\to\End_KV$, $(f,\lambda)\mapsto f\circ r_{\lambda}$ induces a homomorphism of $F$-vector spaces $\varphi:\End_QV\otimes_FK\to\End_KV$ such that $\varphi(f\otimes\lambda)=f\circ r_{\lambda}$ for all $f\in\End_QV$ and all $\lambda\in K$. 
    Clearly $\varphi$ is a $K$-algebra homomorphism. 
    Since $\End_QV\otimes_FK$ is simple and $\varphi(\id_V\otimes1_K)=\id_V\neq0$, it follows that $\varphi$ is injective.
    We have $$\deg\End_QV\otimes_FK=\deg\End_QV=\rk_QV\cdot\deg Q=\rk_KV=\deg\End_KV.$$
    It follows that $\varphi$ is also surjective.
    
    Consider $f\in\End_QV$, $\lambda\in K$ and $v,w\in V$. 
    As $h(v,f(w))=h(\ad_h(f)(v),w)$, we have $h_1(v,f(w))=h_1(\ad_h(f)(v),w)$. 
    Since $h_1$ is $K$-bilinear, we obtain that 
\begin{equation*}
\begin{split}
    h_1((\ad_{h}(f)\circ r_{\lambda})(v),w)&=h_1(\ad_{h}(f)(v\lambda),w)=h_1(\ad_h(f)(v)\lambda,w)\\
    &=h_1(v,f(w)\lambda)=h_1(v,f(w\lambda))=h_1(v,(f\circ r_{\lambda})(w)).
\end{split}
\end{equation*}	
    This shows that $\ad_{h_1}(f\circ r_{\lambda})=\ad_h(f)\circ r_{\lambda}$ for any $\lambda\in K$ and $f\in\End_QV$. 
    This implies that $\varphi$ is an isomorphism of $K$-algebras with involution. Hence
\begin{equation*}
    (\End_KV,\ad h_1)\simeq(\End_QV, \ad h)_K.
\end{equation*}	 
    Since $h_K$ is isotropic, it follows that $h_1$ is isotropic. Fix $v_1\in V\setminus\{0\}$ such that $h_1(v_1,v_1)=0$. 
    Hence $h(v_1,v_1)=h_0(v_1,v_1)$. Since $h_0$ is a skew-hermitian form over $(K,\iota)$, we get that $h(v_1,v_1)=\lambda\alpha$ for some $\lambda\in F$. 
    If $\lambda\neq0$, then we set $\lambda_1=\lambda$ and otherwise we set $\lambda_1=1$.
    In view of \Cref{hermit-isot->hyperplane-subform}, in either case, $\alpha\lambda_1\in\mg{Q}$ and $\alpha\lambda_1$ is represented by $h$.
    By \cite[Lemma 3.6.2]{Knus91}, we get that $h\simeq\langle \alpha\lambda_1\rangle\perp\tilde{h}$ for some non-degenerate skew-hermitian form $\tilde{h}$ over $(Q,\can_Q)$. 
    If $\tilde{h}_K$ is anisotropic, then we set $h'=\tilde{h}$, and otherwise, we repeat the process with $\tilde{h}$ in the place of $h$. 
    This proves the first part of the statement.

    By \Cref{1-dim skew-herm-over-quaternion}, for every $\lambda\in\mg{F}$, the $K$-algebra with involution $\Ad_Q(\la \alpha\lambda\ra)_K$ is hyperbolic.
    It follows by \Cref{isotropic-hyperbolic-inv-equiv-cond} that $\langle\alpha\lambda_1,\ldots,\alpha\lambda_n\rangle_K$ is hyperbolic. \Cref{comput-disc-orth-product-of-quaternions} yields that $e_1(\ad_{\langle\alpha\lambda_1,\ldots,\alpha\lambda_n\rangle})=(-1)^n(\lambda_1\cdots\lambda_n)^2\Nrd_Q(\alpha)^n\sq{F}=(-\Nrd_Q(\alpha))^n\sq{F}$.
\end{proof}

\begin{cor}\label{I^3=0-quad-excellent-skewhermit/Q-split}
    Assume that $\I^3F=0$.
    Let $Q$ be an $F$-quaternion division algebra and let $h$ be an anisotropic skew-hermitian form over $(Q,\can_Q)$. 
    Let $K/F$ be a quadratic field extension such that $Q_K$ is split.
    Then \hbox{$h\simeq h'\perp h''$} for some skew-hermitian forms $h'$ and $h''$ over $(Q,\can_Q)$ such that $h'_K$ is anisotropic, $h''_K$ is hyperbolic and $\rk h''\leq2$.
\end{cor}
\begin{proof}
    By \Cref{quad-ext-excellent-skewhermit/Q-split}, there exist two skew-hermitian forms $h'$ and $h''$ over $(Q,\can_Q)$ such that $h'_K$ is anisotropic, $h''_K$ is hyperbolic and $h\simeq h'\perp h''$. 
    As $h$ is anisotropic, so is $h''$.
    Since $\I^3F=0$, and $(\Ad_Q(h''))_K$ is split and hyperbolic, 
    it follows by \cite[Proposition 4.1]{BQM23} that $\deg\Ad_Q(h'')\leq4$, whereby $\rk h''\leq2$.
\end{proof}

\begin{prop}\label{quad-ext-excellent-inv-division}
    Let $(D,\gamma)$ be an $F$-division algebra with involution.
    Let $\varepsilon\in\{\pm1\}$ and $h$ a non-degenerate $\varepsilon$-hermitian form over $(D,\gamma)$. 
    Let $K/F$ be a quadratic field extension such that $D_K$ is a division algebra. Then $$h\simeq h_1\perp\ldots\perp h_n\perp h'$$ 
    for some $n\in\nat$, $\varepsilon$-hermitian forms $h_1,\ldots,h_n$ over $(D,\gamma)$ of rank $2$ such that $(h_i)_{K}$ is hyperbolic for $1\leq i\leq n$ and some $\varepsilon$-hermitian form $h'$ over $(D,\gamma)$ such that $h'_{K}$ is anisotropic.
\end{prop}
\begin{proof}
    In view of \Cref{hermit-isot->hyperplane-subform}, it is enough to prove the statement for the anisotropic part of $h$.  
    We may therefore assume that $h$ is anisotropic.
    
    If $h_K$ is anisotropic, then we take $n=0$ and $h'=h$. Assume now that $h_K$ is isotropic. 
    Let $V$ be the $D$-right vector space on which $h$ is defined and fix $\alpha\in K\setminus F$. 
    As $h_K$ is isotropic and $K= F(\alpha)=F\oplus F\alpha$, there exist $u,v\in V$ not both zero such that $h_K(u+v\alpha,u+v\alpha)=0$.
    Denote by $W$ the subspace of $V$ spanned by $u$ and $v$ and let $W^\perp$ denote its orthogonal complement with respect to $h$.
    Since $h$ is anisotropic, $h|_{W}$ is a non-degenerate subform of $h$.
    Since $D_{K}$ is a division algebra and $(\ad_{h|_W})_K$ is isotropic, we deduce that $\dim_DW=\dim_{D_{K}} W_{K}=2$. 
    Set $h_1=h|_{W}$. Then $V=W\oplus W^{\perp}$ and $h\simeq h_1\perp h_1'$ with $h_1'=h|_{W^{\perp}}$.
    As $\rk h_1=2$ and $(h_1)_{K}$ is isotropic, we get that $(h_1)_{K}$ is hyperbolic. 
    Repeating this process with $h_1'$ in the place of $h$, we will finally reach a decomposition of $h$ as desired.
\end{proof}

\section{Orthogonal $u$-invariant of quaternion algebras} 
\label{section:orth-hermit-u-inv-quaternion}  

In this section, we obtain an upper bound on the rank of anisotropic skew-hermitian forms defined over a quaternion algebra with canonical involution in terms of the $u$-invariant when $\I^3=0$. 

\begin{prop}\label{P:isotropy-first kind-split-odd deg} 
    Let $(A,\sigma)$ be an $F$-algebra with orthogonal involution such that $\deg A>u(F)$. 
    If $A$ is split, then $\sigma$ is isotropic. In particular, if $\deg A$ is odd, then $\sigma$ is isotropic.
\end{prop}
\begin{proof}
    Since $A$ carries an orthogonal involution, by \cite[Corollary 2.8]{BOI}, $A$ is split when $\deg A$ is odd. 
    
    Assume now that $A$ is split. Then $(A,\sigma)\simeq\Ad(\varphi)$ for some non-degenerate quadratic form $\varphi$ over $F$.
    As $\dim \varphi=\deg A>u(F)$, we have that $\varphi$ is isotropic. Therefore $(A,\sigma)$ is isotropic.    
\end{proof}

\begin{prop}\label{bounds-deg4m-ort-trivialdisc-iso-I^3F=0}
    Assume that $\I^3F=0$. 
    Let $(A,\sigma)$ be an $F$-algebra with orthogonal involution of trivial discriminant with $\ind A\leq2$ and $\deg A\equiv0\bmod 4$.
    If $\deg A>u(F)+2$, then $(A,\sigma)$ is isotropic.
\end{prop} 
\begin{proof}
    If $A$ is split, then the statement follows by \Cref{P:isotropy-first kind-split-odd deg}. 
    Assume now that $\ind A=2$. Hence, $A\sim Q$ for an $F$-quaternion division algebra $Q$. 
    By \Cref{structure-clifford-awi}, we have  $C(A,\sigma)\simeq C^+\times C^-$ for some central simple $F$-algebras $C^+$ and $C^-$ such that $\exp C^+,\exp C^-\leq2$ and $C^+\otimes_FC^-\sim Q$. 
	 
    It follows by \Cref{P:I30-Merk} that $[C^+]=e_2(q^+)$ and $[C^-]=e_2(q^-)$ for two even-dimensional non-degenerate quadratic forms  $q^+$ and $q^-$ over $F$ of trivial discriminant. 
    Recall that $e_2$ vanishes for forms in $\I^3F$. Hence \Cref{ort-sum-rel-e_1-e_2}~$(c)$ yields that $e_2(q^+)$ and $e_2(q^-)$ depend on the forms $q^+$ and $q^-$, respectively, only up to Witt equivalence and similarity.
    We may thus choose $q^+$ and $q^-$ in such way that $\dim q^+=\dim q^-\leq u(F)$ and both $q^+$ and $q^-$ represent $1$. 
    Let $m\in\nat $ be such that $\dim q^+=\dim q^-=2m$. 
    Let $n_Q$ denote the norm form of $Q$, which is a $2$-fold Pfister form over $F$. Let $n'_Q$ denote the $3$-dimensional subform such that 
    $n_Q=\langle1\rangle\perp n_Q'$.
    We write $q^+=\langle1\rangle\perp q^+_0$ and $q^-=\langle1\rangle\perp q^-_0$ for some non-degenerate quadratic forms $q_0^+$, $q_0^-$ over $F$. 
    Note that $q_0^+\perp-q^-\perp-n'_Q$ has trivial discriminant and 
    $$e_2(q_0^+\perp-q^-\perp-n'_Q)=e_2(q^+)+e_2(q^-)+e_2(n_Q)=[C^+]+[C^-]+[Q]=0.$$ 
    Since $\I^3F=0$, we obtain by \Cref{EL:I3=0-classification} that $q^+_0\perp-n'_Q$ is Witt equivalent to $q^-$.  
    Since further $\dim(q^+_0\perp-n'_Q)=2m+2>2m=\dim q^-$, we obtain that $q^+_0\perp -n'_Q$ is isotropic. 
    Hence $q^+_0$ and $n'_Q$ represent some element $a\in F^{\times}$ in common.
    We set $K=F(\sqrt{-a})$. It follows that $(n_Q)_K$ and $q^+_K$ are isotropic. 
    In particular, $Q_K$ is split.
    Since $A\sim Q$, we obtain that $(A,\sigma)_K\simeq\Ad(\varphi)$ for some non-degenerate quadratic form $\varphi$ over $K$ with $\dim\varphi=\deg A$. 
    Furthermore, we have that $e_2(\varphi)=e_2(\sigma_K)=[C^+_K]=e_2(q^+_K)$. 
    Since $\I^3F=0$, we also have $\I^3K=0$, by \cite[Corollary 35.8]{EKM08}. 
    We conclude by \Cref{EL:I3=0-classification} that $\varphi$ is Witt equivalent to $q^+_K$.
	
    Assume now that $\sigma$ is anisotropic. 
    Then $(A,\sigma)\simeq\Ad_Q(h)$ for some anisotropic skew-hermitian form $h$ over $(Q,\can_Q)$ with $\rk h=\frac{1}{2}\deg A$. 
    Since $Q_K$ is split, it follows by \Cref{I^3=0-quad-excellent-skewhermit/Q-split} that $h\simeq h_0\perp h_1$ for two non-degenerate skew-hermitian forms $h_0, h_1$ over $(Q,\can_Q)$ such that $\rk h_0\leq2$, $(h_0)_K$ is hyperbolic, and $(h_1)_K$ is anisotropic. 
    We obtain that $$\Ad(\varphi)\simeq(A,\sigma)_K\simeq(\Ad_Q(h_0\perp h_1))_K$$
    and that $\varphi\simeq \psi\perp (\rk h_0)\times\hh$ for some anisotropic quadratic form $\psi$ over $K$.
    Hence $\deg A=\dim\varphi =\dim\psi+2\rk h_0\leq \dim\psi+4$.
    Recall that $q^+_K$ is isotropic and Witt equivalent to $\varphi$, hence also to $\psi$.
    Therefore $$\dim \psi \leq \dim q^+_K-2=2m-2.$$ 
    We conclude that $\deg A\leq 2m+2\leq u(F)+2$.
\end{proof}

\begin{prop}\label{C:trivial disc-(n-m)dim-subform of skew-herm}
    Assume that $\I^3F=0$ and $u(F)<\infty$. Let $Q$ be an $F$-quaternion division algebra and $m=\lceil \frac{1}4 u(F)\rceil$.
    Let $h$ be a non-degenerate skew-hermitian form over $(Q,\can_Q)$. Then the following hold:
\begin{enumerate}[$(a)$]
    \item If $\rk h=2n$ for some integer $n>m$ and $\ad_h$ has trivial discriminant, then $h\simeq (n-m)\times \hh \perp h'$ for some non-degenerate skew hermitian form $h'$ over $(Q,\can_Q)$.
    \item If $\rk h\geq 2m+2$, 
    then $h\simeq h_0\perp h_1$ for two non-degenerate skew-hermitian forms $h_0$ and $h_1$ over $(Q,\can_Q)$ such that $\rk h_0=2m$ and $e_1(\ad_{h_0})$ is trivial.	
    \item If $\rk h\geq 2m+4$, then $h$ is isotropic. 
\end{enumerate}
\end{prop}
\begin{proof}
    If $\rk h=2n+1$ for some integer $n\geq m$, then to show $(b)$ or $(c)$, we may replace $h$ by a non-degenerate subform of rank $2n$. 
    Hence we may assume throughout without loss of generality that $\rk h=2n$ for an integer $n\geq m$.

    Assume first that $\ad_h$ has trivial discriminant.
    Then $\Ad_Q(h)$ is an $F$-algebra with orthogonal involution of trivial discriminant. 
    If $n>m$, then it follows by \Cref{bounds-deg4m-ort-trivialdisc-iso-I^3F=0} that $\Ad_Q(h)$ is isotropic, whereby $h$ is isotropic and hence contains $\hh$ as a subform, by \Cref{hermit-isot->hyperplane-subform}.
    An $(n-m)$-fold application of this argument establishes $(a)$.

    To show $(b)$ and  $(c)$, we fix $d\in F^{\times}$ with $e_1(\ad_{h})=d F^{\times2}$.  
    Using \Cref{exists-orthogonal-given-disc-I^3=0}, we take a non-degenerate skew-hermitian form $h''$ of rank $2$ over $(Q,\can_Q)$ with $e_1(\ad_{h''})=d F^{\times2}$.
    We consider the non-degenerate skew-hermitian form $h'=h\perp h''$ over $(Q,\can_Q)$. 
    Then $\rk h'=2n+2$, and by \Cref{ort-sum-rel-e_1-e_2} ~$(b)$, $e_1(\ad_{h'})$ is trivial. 
    Hence, we obtain by $(a)$ that $$h'\simeq(n+1-m)\times\hh\perp h_0$$ 
    for a non-degenerate skew-hermitian form $h_0$ over $(Q,\can_Q)$ with $\rk h_0=2m$. 
    As $e_1(\ad_{h'})$ is trivial, we obtain by \Cref{ort-sum-rel-e_1-e_2}~$(b)$ that $e_1(\ad_{h_0})$ is trivial.
    Set $h_1=-h''$. Then 
    $$h\perp2\times\hh\simeq h'\perp h_1\simeq(n+1-m)\times\hh\perp h_0\perp h_1\,.$$ 
    
    Assume now that $n\geq m+1$. 
    Using Witt Cancellation \cite[Corollary 6.4.2]{Knus91}, we get that $$h\simeq(n-1-m)\times\hh\perp h_0\perp h_1.$$
    This establishes $(b)$. Furthermore, if $n>m+1$, then $h$ is isotropic, and this shows $(c)$.
\end{proof}

\begin{thm}\label{bound-anis-alg-ind2-orth-I^3=0}
    Assume that $\I^3F=0$. Let $(A,\sigma)$ be an $F$-algebra with orthogonal involution with $\ind A\leq2$. 
    If $\deg A>4\lceil\frac{u(F)}{4}\rceil+6$, then $(A,\sigma)$ is isotropic.
\end{thm} 
\begin{proof}
    If $\ind A=1$, then \Cref{P:isotropy-first kind-split-odd deg} implies that $\s$ is isotropic. Assume now that $\ind A=2$. 
    Then $A$ is Brauer equivalent to an $F$-quaternion division algebra $Q$. 
    We have $(A,\sigma)\simeq\Ad_Q(h)$ for some non-degenerate skew-hermitian form $h$ over $(Q,\can_Q)$ with $\rk h=\frac{1}{2}\deg A$.
    Thus $\rk h>2\lceil\frac{u(F)}{4}\rceil+3$, and it follows by \Cref{C:trivial disc-(n-m)dim-subform of skew-herm} that $h$ is isotropic.
    Hence $(A,\sigma)$ is isotropic.
\end{proof}

\begin{prop}\label{iso-deg4m-nontrivial-e1}
    Assume that $\I^3F=0$. Let $(A,\sigma)$ be an $F$-algebra with orthogonal involution such that $\ind A\leq2$. 
    Assume that $\deg A=4\lceil\frac{u(F)}{4}\rceil+4$.
    Let $d\in F^{\times}$ be such that $e_1(\sigma)=d F^{\times2}$ and set $K=F(\sqrt{d})$. 
    Assume that $u(K)\leq4\lceil\frac{u(F)}{4}\rceil$. Then $(A,\sigma)$ is isotropic.
\end{prop}
\begin{proof}
    If $\ind A=1$, then the statement follows by \Cref{P:isotropy-first kind-split-odd deg}. 
    Assume that $\ind A=2$. Then $A$ is Brauer equivalent to an $F$-quaternion division algebra $Q$.
    If $d\in \sq{F}$, then \Cref{bounds-deg4m-ort-trivialdisc-iso-I^3F=0} implies that $(A,\sigma)$ is isotropic. 
    So we may assume that $d\in F^{\times}\setminus F^{\times2}$, whereby $K/F$ is a quadratic field extension.
	
    Set $n=\lceil\frac{u(F)}{4}\rceil+1$. Then $\deg A=4n\geq u(K)+4$. 
    Since $\I^3 F=0$, it follows by \cite[Theorem 34.22]{EKM08} that $\I^3K=0$.
    As $\sigma_K$ has trivial discriminant, it follows by \Cref{bounds-deg4m-ort-trivialdisc-iso-I^3F=0} that $\sigma_K$ is isotropic.
   
    We fix a non-degenerate skew-hermitian form $h$ over $(Q,\can_Q)$ with $\rk h=2n$ such that  $(A,\sigma)\simeq\Ad_Q(h)$. Then $h_K$ is isotropic.
    We claim that $h\simeq h_0\perp h_1$ for some non-degenerate skew-hermitian forms $h_0$ and $h_1$ over $(Q,\can_Q)$ with $\rk h_0=2$, $\rk h_1=2n-2$ and such that $(h_0)_{K}$ is hyperbolic.
    If $Q_K$ is split, this follows by \Cref{I^3=0-quad-excellent-skewhermit/Q-split}, using that $\rk h\geq\frac{u(K)}{2}+2$. 
    If $Q_K$ is a division algebra, this follows by \Cref{quad-ext-excellent-inv-division}.
    Since $Q$ is a division algebra, $\rk h_0=2$ and $(h_0)_K$ is hyperbolic, we obtain by \Cref{quad-ext-hyper->trivial-disc} that $e_1(\ad_{h_0})$ is trivial.
	
    We now set $(A_i,\sigma_i)=\Ad_Q(h_i)$ for $i=0,1$. 
    It follows that $\deg A_0=4$, $\deg A_1=4n-4$, $A_0\sim A_1\sim Q$ and the involutions $\sigma_0$ and $\sigma_1$ are orthogonal.
    Furthermore $e_1(\sigma_0)=e_1(\ad_{h_0})=\sq{F}$. 
    It follows by \Cref{ort-sum-rel-e_1-e_2} $(b)$ that $e_1(\sigma_1)=e_1(\ad_{h_1})=e_1(\ad_{h})=e_1(\sigma)=d F^{\times2}$. 
    By \Cref{app-disc-deg2-4-6}, we have that $(A_0,\sigma_0)\simeq(H,\can_H)\otimes_F(H',\can_{H'})$ for some $F$-quaternion algebras $H$ and $H'$.
    In view of \Cref{clifford-ort-inv-biquat}, we have $e_2(\ad_{h_0})=e_2(\sigma_0)=[H]$ in $\Br(F)/\langle [Q]\rangle$.
    Then $H\otimes_FH'\simeq A_0\sim Q$, and since $(A_0,\sigma_0)_K$ is hyperbolic, it follows by \Cref{iso=hyper-for-deg4+orth+trivialdisc} that one of $H_K$ and  $H'_K$ is split. 
    We may assume without loss of generality that $H_K$ is split. 
    Hence $H\simeq(d,\lambda)_F$ for some $\lambda\in F^{\times}$.
    We consider the $F$-algebra with orthogonal involution $$\Ad_Q(h_1\perp-\lambda h_1)\in(A_1,\sigma_1)\boxplus(A_1,\sigma_1)\,.$$ 
    Its degree is $8n-8$, and by \Cref{ort-sum-rel-e_1-e_2}~$(b)$, its discriminant is trivial.
    Since $h_1\perp-\lambda h_1=\la 1, -\lambda\ra\otimes h_1$ and $e_1(\ad_{h_1})=d\sq{F}$, we conclude that $e_2(\ad_{h_1\perp-\lambda h_1})=[(d,\lambda)]$ in $\Br(F)/\la[Q]\ra$.
    We set $h'=h_0\perp h_1\perp-\lambda h_1$ and note that $\rk h'=4n-2$. 
    We consider the $F$-algebra with orthogonal involution $$\Ad_Q (h')\in(A_0,\sigma_0)\boxplus(A_1,\sigma_1)\boxplus(A_1,\sigma_1).$$ 
    By \Cref{ort-sum-rel-e_1-e_2}, we have that $e_1(\ad_{h'})$ is trivial and that
    $$e_2(\ad_{h'})=e_2(\ad_{h_0})+e_2(\ad_{h_1\perp-\lambda h_1})=[H]+[H]=0\quad\text{in}\,\Br(F)/\langle[Q]\rangle.$$
    Since $\I^3F=0$, it follows by \Cref{deg4m-ortogonal-isot-hyper-I^3F=0} that $h'$ is hyperbolic. 
    Then, since $\rk(h_0\perp h_1)=2n>\frac{1}{2}\rk h'$, we conclude by \Cref{half-rk-subform-hyper->iso} that $h_0\perp h_1$ is isotropic. 
    Therefore $h$ is isotropic. Hence $\sigma$ is isotropic.
\end{proof}

\begin{thm}\label{bound-anis-alg-ind2-orth-I^3=0+uK<uF}
    Assume that $\I^3F=0$. Let $k\in\nat$ be such that $u(K)\leq 4k$ for every field extension $K/F$ with $[K:F]\leq2$. 
    Let $A$ be a central simple $F$-algebra with $\ind A\leq 2$ and $\deg A>4k+2$.
    Then every orthogonal involution on $A$ is isotropic.
\end{thm}
\begin{proof}
    If $\ind A=1$, then the statement follows by \Cref{P:isotropy-first kind-split-odd deg}.
    Assume now that $\ind A=2$. Then $A$ is Brauer equivalent to an $F$-quaternion division algebra $Q$.
    Let $\s$ be an orthogonal involution on $A$.
    Then $(A,\s)\simeq \Ad_Q(h)$ for a non-degenerate skew-hermitian form $h$ over $(Q,\can_Q)$ with $\rk h= \frac{1}2\deg A>2k+1$.
    If $\deg A>4\lceil\frac{u(F)}{4}\rceil+6$, then it follows by \Cref{bound-anis-alg-ind2-orth-I^3=0} that $\s$ is isotropic. 
    Assume now that $\deg A\leq 4\lceil \frac{u(F)}4\rceil+6$.
    Since we have $u(F)\leq 4k$, we obtain that $4k+2< \deg A\leq 4\lceil \frac{u(F)}4\rceil+6 \leq 4k+6$.
    In particular $\lceil\frac{u(F)}4\rceil=k$ and  $\rk h\in\{2k+2,2k+3\}$.
    We choose a non-degenerate subform $h'$ of $h$ with $\rk h'=2k+2$ and consider the $F$-algebra $(A',\s')=\Ad_Q(h')$.
    We fix $d\in\mg{F}$ such that $e_1(\s')=d\sq{F}$ and set $K=F(\sqrt{d})$.
    Then $u(K)\leq 4k$, and it follows by \Cref{iso-deg4m-nontrivial-e1} that $\s'$ is isotropic.
    We obtain that $h'$ is isotropic, and hence so is $h$. Therefore $\s$ is isotropic.
\end{proof}

\begin{cor}\label{improv-u-inv6-bounds-orth-ind2-anisot}
    Assume that $u(F)\leq6$. 
    Let $(A,\sigma)$ be an $F$-algebra with orthogonal involution such that $\sigma$ is anisotropic and $\ind A\leq2$. 
    Then $\deg A\leq10$.
\end{cor}
\begin{proof}
    Consider a quadratic field extension $K/F$. 
    We have $u(K)\leq\frac{3}{2} u(F)\leq9$, by \cite[Theorem 4.3]{EL73}.  
    Since $u(F)\leq 6$, in particular all $3$-fold Pfister forms are hyperbolic and thus $\I^3F=0$. 
    Hence also $\I^3K=0$, by \cite[Theorem 34.22]{EKM08}. 
    Then $u(K)$ is even, by \cite[Proposition 36.3]{EKM08}. Therefore $u(K)\leq 8$. 
    
    This argument shows that $u(K)\leq 8$ for every field extension $K/F$ with $[K:F]\leq 2$.
    We conclude by \Cref{bound-anis-alg-ind2-orth-I^3=0+uK<uF} that $\deg A\leq 10$.
\end{proof}

Let us reformulate the main results of this section in terms of the orthogonal $u$-invariant.
Consider a central simple $F$-algebra $B$ that carries an orthogonal involution $\gamma'$.
The \emph{orthogonal $u$-invariant of} $B$, denoted by $u^+(B)$, is defined as follows:
\begin{equation*}
    \mbox{$u^+(B)=\sup\{\rk h\mid h\,\text{anisotropic hermitian form over}\,(B,\gamma')\}$}.
\end{equation*} 
This definition does not depend on the particular choice of $\gamma'$; see \cite[Proposition 2.2]{Mah05}. 
For $B=F$, we obtain that $u^+(F)=u(F)$, as non-degenerate hermitian forms with respect to the orthogonal involution $\id_F$ on $F$ are given by the polar forms of non-degenerate quadratic forms.

Taking instead of an orthogonal involution a symplectic involution $\gamma$ on $B$, provided that such an involution exists, we can also express $u^+(B)$ in terms of skew-hermitian forms over $(B,\gamma)$:
\begin{eqnarray*}
    \mbox{$u^+(B)$\,}&\mbox{$=\sup\{\rk h\mid h\,\text{anisotropic skew-hermitian form over}\,(B,\gamma)\}$}.
\end{eqnarray*}
We will now use this latter description of $u^+(B)$ in the special situation where $B$ is an $F$-quaternion division algebra, by taking $\gamma=\can_B$.

\begin{cor}\label{orthogonal-hermit-u-quater-I^3=0}\label{improv-u-inv6-orth-hermit-u-inv-quat}
Let $Q$ be an $F$-quaternion division algebra.
\begin{enumerate}[$(a)$ ]
    \item If $\I^3 F=0$, then $u^+(Q)\leq2\lceil\frac{u(F)}{4}\rceil+3$.
    \item If $u(F)\leq 6$, then $u^+(Q)\leq 5$.
\end{enumerate} 
\end{cor}
\begin{proof}
    Part $(a)$ follows directly by \Cref{C:trivial disc-(n-m)dim-subform of skew-herm}, and $(b)$ is clear by \Cref{improv-u-inv6-bounds-orth-ind2-anisot}.
\end{proof}

\begin{rem}\label{remark:improv-u-inv6-hermit-quat}
    Let $Q$ be an $F$-quaternion division algebra. By \cite[Corollary 3.4]{Mah05}, the bound $u^+(Q)\leq\frac{5}{4}u(F)$ holds without any condition on $\I^3 F$. 
    This covers part $(a)$ from \Cref{orthogonal-hermit-u-quater-I^3=0} in the case where $u(F)=4$.
    When $\I^3F=0$ and $u(F)\geq 6$, then the bounds from \Cref{orthogonal-hermit-u-quater-I^3=0} are better than what can be obtained from \cite{Mah05}.
\end{rem}

\section{Anisotropic orthogonal involutions in degree $8$}\label{section:anisotropic-inv-deg8-I^3=0} 

In this section, we provide examples of anisotropic orthogonal involutions on a central simple algebra of degree $8$ over a field with $u$-invariant $4$. 
The discriminant of such an involution is necessarily non-trivial, in view of \Cref{bounds-deg4m-ort-trivialdisc-iso-I^3F=0}. 
We first collect some well-known facts, mostly coming from \cite{EL73}.

\begin{prop}\label{P:u4-Br2}
    Assume that $u(F)\leq 4$. Let $A$ be a central simple $F$-algebra with $\exp A\leq 2$.
    Then $A\sim Q$ for an $F$-quaternion algebra $Q$.
\end{prop}
\begin{proof}
    It follows by \Cref{P:I30-Merk} that 
    $A\sim C(\varphi)$ for an even-dimensional non-degenerate quadratic form $\varphi$ over $F$ of trivial discriminant.
    Furthermore, as $u(F)\leq 4$, we may choose $\varphi$ such that $\dim \varphi=4$.
    Then $C(\varphi)\simeq \mathbb{M}_{2}(Q)$ for an $F$-quaternion algebra $Q$.
    We obtain that $A\sim C(\varphi)\sim Q$.
\end{proof}

\begin{lem}\label{P:u6-biquatexist}
Assume that $u(F)=6$. Then there exists an $F$-biquaternion division algebra. 
\end{lem}
\begin{proof}
    By \cite[Proposition 1.4]{EL73}, there exists some $6$-dimensional anisotropic quadratic form $\varphi$ over $F$ of trivial discriminant.
    By \cite[Section 16.A]{BOI}, we obtain that $C(\varphi)\simeq\mathbb{M}_2(B)$ for an $F$-biquaternion division algebra $B$. 
\end{proof}

\begin{lem}\label{descent-qf-quadext-finite&even-u-inv} 
    Let $K/F$ be a quadratic field extension and $\varphi$ a non-degenerate quadratic form over $K$. 
    Then $\varphi\simeq \vartheta_K\perp\psi$ for a non-degenerate quadratic form $\vartheta$ over $F$ and a non-degenerate quadratic form $\psi$ over $K$ with $\dim \psi\leq \frac{1}2u(F)$.
\end{lem}
\begin{proof}
    Since $\car F\neq 2$ we can reason with the polar form of $\varphi$.
    The statement now follows by choosing an $F$-linear map $s:K\to F$ with $\Ker s=F$ and then applying \cite[Proposition 34.1]{EKM08}.
\end{proof}

\begin{prop}\label{L:u-inv6-div-biquat}
    Assume that $u(F)\leq4$. Let $K/F$ be a quadratic field extension. Then either $u(K)\leq 4$ or $u(K)=6$. 
    Furthermore, any central simple $K$-algebra of exponent at most $2$ is Brauer equivalent to $(a,b)_K\otimes_K(c,x)_K$ for some $a,b,c\in F^{\times}$ and $x\in K^{\times}$.
\end{prop}
\begin{proof}
    Since $u(F)\leq 4$, it follows by \cite[Theorem 4.3]{EL73} that $u(K)\leq 6$. 
    Since $u(K)\neq5$ by \cite[Corollary 36.4]{EKM08}, we conclude that $u(K)\leq 4$ or $u(K)=6$.

    Let $B$ be a central simple $K$-algebra with $\exp B\leq 2$. 
    It follows by \Cref{P:I30-Merk} that $B\sim C(\varphi)$ for some even-dimensional non-degenerate quadratic form $\varphi$ over $K$ of trivial discriminant, and since $u(K)\leq 6$, we may choose $\varphi$ such that $\dim \varphi=6$.
    Since $u(F)\leq 4$, by \Cref{descent-qf-quadext-finite&even-u-inv} we may write $\varphi\simeq \vartheta_K\perp\psi$ with a non-degenerate $4$-dimensional quadratic form $\vartheta$ over $F$ and a non-degenerate $2$-dimensional quadratic form $\psi$ over $K$.
    We choose elements $u,v,w,t\in \mg{F}$ such that $\vartheta\simeq \la u,v,w,t\ra$.
    Let $y\in\mg{K}$ be an element represented by $\psi$.
    As $\vf$ has trivial discriminant, we obtain that $\vf\simeq\la u,v,w,t,y,-tuvwy\ra$.
    Using \cite[Proposition 6.4.7 and Corollary 6.2.7]{Kahn08}, one  finds that $C(\varphi)\simeq(-uv,-vw)_K\otimes_K(uvwt,-ty)_K\otimes_K\mathbb{M}_2(K)$.
    Now we set $a=-uv$, $b=-vw$, $c=uvwt$ and $x=-ty$, whereby $a,b,c\in\mg{F}$, $x\in\mg{K}$ and $B\sim C(\varphi)\sim  (a,b)_K\otimes_K(c,x)_K$.
\end{proof}

\begin{prop}\label{anisot-deg8-ind2-u4}
    Let $d\in F^{\times}\setminus F^{\times2}$ and $K=F(\sqrt{d})$. Assume that $u(F)\leq4$ and $u(K)=6$.
    Then there exists an anisotropic $F$-algebra with orthogonal involution $(A,\sigma)$ with $\deg A=8$, $\ind A=2$ and $e_1(\sigma)=d F^{\times2}$.
\end{prop}
\begin{proof}
    By \Cref{P:u6-biquatexist}, there exists a $K$-biquaternion division algebra $B$.
    By \Cref{L:u-inv6-div-biquat}, there exist $a,b,c\in\mg{F}$ and $x\in\mg{K}$ such that 
    $$B\simeq (a,b)_K\otimes_K (c,x)_K\,.$$
    We set $H=(a,b)_F$ and $Q=(c,x)_K$.
    Let $\iota$ be the non-trivial $F$-automorphism of $K$ and ${}^{\iota}Q$ the $K$-quaternion algebra obtained by letting $K$ act on $Q$ via~$\iota$.
    Now \cite[Section 15.B]{BOI} provides an $F$-biquaternion algebra with orthogonal involution $(C,\sigma_1)$ such that $(C,\sigma_1)_K\simeq (Q\otimes_K\!{^{\iota}}Q,\can_{Q}\otimes \can_{\,^{\iota}Q})$, $e_1(\sigma_1)=d F^{\times2}$ and $C\simeq \mathbb{M}_{2}(Q')$ for $Q'=(c,\mathsf{N}_{K/F}(x))_F$, where $\mathsf{N}_{K/F}(x)$ is the norm of $x$ with respect to $K/F$.
 
    We claim that $Q'$ is a division algebra.
    Suppose for the sake of a contradiction that $Q'$ is split.
    Then $\mathsf{N}_{K/F}(x)$ is represented by the quadratic form $\la 1,-c\ra$ over $F$.
    It follows by \cite[Norm Principle 2.13]{EL76} that $x=ey$ for some $y\in\mg{K}$ which is represented by $\la 1,-c\ra$ over $K$ and some $e\in\mg{F}$. 
    Then $Q\simeq (c,e)_K$ and $B\simeq H_K\otimes_K Q=(a,b)_K\otimes_K (c,e)_K$.
    Since $B$ is a division algebra, we conclude that $(a,b)_F\otimes_F (c,e)_F$ is a division algebra.
    By \Cref{P:u4-Br2}, this contradicts the hypothesis that $u(F)\leq 4$. Hence $Q'$ is a division algebra.
	
    Since $u(F)\leq4$, it follows by \Cref{P:u4-Br2} that $H\otimes_FQ'\sim H'$ for an $F$-quaternion algebra $H'$.
    Then $H\otimes_FH'\sim Q'\sim C$. As $\deg H\otimes_FH'=4=\deg C$, we obtain that $H\otimes_FH'\simeq C$.
    We fix an orthogonal involution $\sigma_0$ on $C$ such that $(H\otimes_F H',\can_H\otimes\can_{H'})\simeq (C,\sigma_0)$.
    Then $e_1(\sigma_0)$ is trivial, by \Cref{app-disc-deg2-4-6}, and $e_2(\sigma_0)=[H]$ in $\Br(F)/\langle[Q']\rangle$, by \Cref{clifford-ort-inv-biquat}.
    For $i=0,1$, we have that $(C,\sigma_i)\simeq\Ad_{Q'}(h_i)$ for some non-degenerate skew-hermitian form $h_i$ over $(Q',\can_{Q'})$ with $\rk h_i=2$.
    
    We consider now the $F$-algebra with orthogonal involution $$(A,\sigma)\,\,=\,\,\Ad_{Q'}(h_0\perp h_1)\,.$$
    Note that $\deg A=8$ and $A\sim Q'$. In particular $\ind A=2$.
    Since $e_1(\sigma_0)$ is trivial and $(A,\sigma)\in (C,\sigma_0)\boxplus(C,\sigma_1)$, we obtain by \Cref{ort-sum-rel-e_1-e_2} $(b)$ that $e_1(\sigma)=e_1(\sigma_1)=d F^{\times2}$.

    We claim that $(A,\sigma)$ is anisotropic. Establishing this will finish the proof.
    Suppose on the contrary that $(A,\sigma)$ is isotropic. 
    Then, by \Cref{L2}, there exists a quadratic field extension $L/F$ such that $(C,\sigma_i)_L$ is split and isotropic for $i=0,1$. 
    Since $(H_L\otimes^{\vphantom{M}}_L H'_L,\can_{H_L}\otimes\can_{H'_L})\simeq(C,\sigma_0)_L$, which is split and isotropic,
    it follows by \Cref{iso=hyper-for-deg4+orth+trivialdisc} that  $H_L$ and $H'_L$ are split. 
    Since furthermore $(Q_{KL},\can_{Q_{KL}})\otimes_L(^{\iota}Q_{KL},\can_{^{\iota}Q_{KL}})\simeq (C,\sigma_1)_{KL}$, which is isotropic, \Cref{iso=hyper-for-deg4+orth+trivialdisc} yields that $Q_{KL}$ is split. 
    In particular $B_{KL}$ is split. Hence $\ind B\leq[KL:K]\leq2$.
    This contradicts that $B$ is a division algebra. 
\end{proof}

\begin{exmp}
    By \cite{LM94}, there exists a field $F$ with $\car F\neq2$ and $u(F)=4$ having a quadratic field extension $K/F$ with $u(K)=6$. 
    (See \cite[Example 4.5]{BB23} for another construction of such an example.)
    It follows by \Cref{anisot-deg8-ind2-u4} that there exists an $F$-algebra with orthogonal involution $(A,\sigma)$ with $\ind A=2$, $\deg A=8$ and such that $\sigma$ is anisotropic.
\end{exmp}

In the following, we give a necessary and sufficient condition for a  field with $u$-invariant $4$ to admit a quadratic field extension with $u$-invariant $6$.

\begin{thm}\label{T:exists-quadext-u6}
    Assume that $u(F)\leq4$. Let $d\in \mg F\setminus\sq F$ and $K=F(\sqrt{d})$.
    If $u(K)=6$, then there exists an anisotropic $F$-algebra with orthogonal involution $(A,\sigma)$ with $\deg A=8$, $\ind A=2$ and  $e_1(\sigma)=d F^{\times2}$.
    Otherwise $u(K)\leq 4$, and every $F$-algebra with orthogonal involution of degree $8$ of discriminant $d\sq{F}$ is isotropic.
\end{thm}
\begin{proof}
    By \Cref{L:u-inv6-div-biquat}, we have $u(K)\leq 4$ or $u(K)=6$. 
    If $u(K)\leq 4$, then it follows by \Cref{iso-deg4m-nontrivial-e1} that every $F$-algebra with orthogonal involution of degree $8$ and discriminant $d\sq{F}$ is isotropic.
    If $u(K)=6$, then by \Cref{anisot-deg8-ind2-u4} there exists an anisotropic $F$-algebra with orthogonal involution of degree $8$, index $2$ and  discriminant $d\sq{F}$.
\end{proof}

\section{Anisotropic orthogonal involutions in degree $10$} 
\label{section:anisotropic-inv-deg10-I^3=0}   

In this section, we show that the bounds obtained in \Cref{bound-anis-alg-ind2-orth-I^3=0} and \Cref{orthogonal-hermit-u-quater-I^3=0} are optimal when $u(F)=4$.

\begin{thm}\label{deg10-ort-iso-equiv-cond}
    Assume that $u(F)\leq 4$. The following are equivalent:
\begin{enumerate}[$(1)$]
    \item Every $F$-algebra with orthogonal involution of degree $10$ with trivial discriminant is isotropic.
    \item Every central simple $F$-algebra of exponent $4$ has index $4$.
    \item For every central $F$-division algebra $B$ of degree $4$ and every $F$-quaternion algebra $H$, there exists a quadratic field extension $K/F$ such that $H_K$ is split and $\ind B_K=2$.
\end{enumerate}
\end{thm}
\begin{proof}
    $(1\Rightarrow 3)$ Consider a central $F$-division algebra $B$ with $\deg B=4$ and an $F$-quaternion algebra $H$.  
    Since $B$ is not Brauer equivalent to any $F$-quaternion algebra, it follows by \Cref{P:u4-Br2} that $\exp B>2$. 
    Since $\exp B$ divides $\deg B$, we conclude that $\exp B=4$.  Hence $\exp B^{\otimes 2}=2$.
    Since $u(F)\leq 4$, it follows by \Cref{P:u4-Br2} that there exist $F$-quaternion algebras $Q$ and $H'$ such that $B^{\otimes 2}\sim Q$ and $H\otimes_FQ\sim H'$. 
    As $\exp B=4$, we have that $Q$ is an $F$-division algebra.
    We set $$(A_0,\sigma_0)=(H,\can_H)\otimes (H',\can_{H'})$$
    and observe that $A_0\sim Q$ and $\sigma_0$ is an orthogonal involution of trivial discriminant.
    Let $\sw$ denote the switch involution on $B\times B^{\op}$. This is a unitary involution.
    As explained in \cite[Section 15.D]{BOI}, the construction of the discriminant algebra of $(B\times B^{\op},\sw)$ and of its induced canonical involution 
    provides us with an $F$-algebra with orthogonal involution $(A_1,\sigma_1)$ of degree $6$  and trivial discriminant having the property that
    $$C(A_1,\sigma_1)\simeq B\times B^{\op}\,.$$
    By \Cref{structure-clifford-awi}, we obtain that $A_1\sim B^{\otimes2}\sim Q$. 
    As $A_0\sim A_1\sim Q$, there exist two non-degenerate skew-hermitian forms $h_0$ and $h_1$ over $(Q,\can_Q)$ such that 
    $$\hspace{2.5cm}(A_i,\sigma_i)\simeq\Ad_Q(h_i) \quad \mbox{ for }i=0,1\,.$$
    We obtain that $\rk h_0=\frac{1}2\deg A_0=2$ and $\rk h_1=\frac{1}2\deg A_1=3$. 
    We now consider the $F$-algebra with orthogonal involution $$(A,\sigma)=\Ad_Q(h_0\perp h_1)\,.$$
    We have $\deg A=10$, $A\sim Q$ and $(A,\sigma)\in (A_0,\sigma_0) \boxplus(A_1,\sigma_1)$, and we compute by \Cref{ort-sum-rel-e_1-e_2} $(b)$ that $\sigma$ has trivial discriminant.
    
    Assume now that Condition $(1)$ holds. Then $(A,\sigma)$ is isotropic. 
    It follows by \Cref{L2} that there exists a quadratic field extension $K/F$ such that $(A_0,\sigma_0)_K$ and $(A_1,\sigma_1)_K$ are both split and isotropic.
    We obtain that $e_2((\sigma_1)_K)=[B_K]$ in $\Br(K)$, the $K$-quaternion algebras $H_K, H'_K, Q_K$ are all split and $(A_1,\sigma_1)_K\simeq\Ad(\varphi)$ for some non-degenerate, isotropic $6$-dimensional quadratic form $\varphi$ over $K$ of trivial discriminant. 
    As $[B_K]=e_2((\sigma_1)_K)=e_2(\varphi)$ and $\varphi$ is isotropic, we obtain that $\ind B_K\leq 2$.
    Since $\ind B=4$ and $[K:F]=2$, we conclude that $\ind B_K=2$.
    Hence we have found a quadratic field extension $K/F$ such that $H_K$ is split and $\ind B_K=2$.
\smallskip

    $(3 \Rightarrow 2)$ Consider a central simple $F$-algebra $B$ with $\exp B=4$. 
    Then $\ind B\geq 4$ and $\exp B^{\otimes2}=2$.
    Since $u(F)\leq4$, we obtain by \Cref{P:u4-Br2} that $\ind B^{\otimes2}=2$. 
    Hence there exists a quadratic field extension $K/F$ such that $B^{\otimes2}_K$ is split.
    Since $u(F)\leq 4$, it follows by \Cref{L:u-inv6-div-biquat} that there exist an $F$-quaternion algebra $H$ and a $K$-quaternion algebra $Q$ such that
    $B_K\sim H_K\otimes_K Q$.
    Let $B'$ denote the central $F$-division algebra which is Brauer equivalent to $B\otimes_FH$. 
    Since $\exp H\leq 2$, we obtain that $\exp B'=\exp B=4$.
    Furthermore $B'_K\sim Q$, whereby $\ind B'\leq [K:F]\cdot \ind Q\leq 4=\exp B'$.
    We conclude that $\deg B'=\ind B'=\exp B'=4$.

    Assume now that $(3)$ holds. 
    Then there exists a quadratic field extension $L/F$ such that $\ind B'_L=2$ and $H_L$ is split. 
    Then $B_L\sim B'_L$ and in particular $\ind B_L=\ind B'_L=2$. 
    We obtain that $\ind B\leq [L:F]\cdot \ind B_L\leq 4=\exp B$, whereby $\ind B=\exp B =4$.
\smallskip

    $(2\Rightarrow 1)$\, Let $(A,\sigma)$ be an $F$-algebra with orthogonal involution of degree $10$ and with trivial discriminant.
    As $\exp A\leq2$ and $\deg A=10$, we get that $\ind A\leq2$. 
    If $A$ is split, then as $\deg A>u(F)$, it follows by \Cref{P:isotropy-first kind-split-odd deg} that $(A,\sigma)$ is isotropic.
    Assume now that $\ind A=2$. 
    Let $Q$ denote the $F$-quaternion division algebra which is Brauer equivalent to $A$. 
    We have that $$(A,\sigma)\simeq\Ad_Q(h)$$ for a non-degenerate skew-hermitian form $h$ over $(Q,\can_Q)$ with $\rk h=5$.
    By \Cref{C:trivial disc-(n-m)dim-subform of skew-herm}, we have that $h\simeq h_0\perp h_1$ for two non-degenerate skew-hermitian forms $h_0$ and $h_1$ over $(Q,\can_Q)$ with $\rk h_0=2$, $\rk h_1=3$ and such that $\ad_{h_0}$ has trivial discriminant. 
    For $i=1,2$, we set $$(A_i,\sigma_i)=\Ad_Q(h_i).$$
    Then $\sigma_0$ has trivial discriminant, and since $\deg A_0=\deg Q\cdot \rk h_0=4$, we get by \Cref{app-disc-deg2-4-6} that $$(A_0,\sigma_0)\simeq(H,\can_H)\otimes(H',\can_{H'})$$ 
    for two $F$-quaternion algebras $H$ and $H'$.
    It follows that $H\otimes_FH'\sim A_0\sim Q$. Furthermore, $\deg A_1=6$ and $A_1\sim Q$. 
    Since $\s$ and $\s_0$ have trivial discriminant, so has $\s_1$, in view of \Cref{ort-sum-rel-e_1-e_2} $(b)$.
    Hence, it follows by \Cref{structure-clifford-awi} that $C(A_1,\sigma_1)\simeq B\times B^{\op}$ for some central simple $F$-algebra $B$ with $\deg B=4$ such that $B^{\otimes2}\sim A_1\sim Q$.  
    Since $\exp Q=\ind Q=2$, we obtain that $\exp B=4$. 
    Since $\exp H\leq 2$, we conclude that $\exp B\otimes_FH=\exp B=4$. 
    
    Assume now that $(2)$ holds. 
    Then $B\otimes_FH\sim B'$ for some central $F$-division algebra $B'$ of degree $4$. 
    Let $\sw$ denote the switch involution on $B'\times B'^{\op}$ and let $(A_2,\sigma_2)$ denote the discriminant algebra of $(B'\times B'^{\op},\sw)$ with its canonical involution. 
    By \cite[Section 15.D]{BOI}, $(A_2,\sigma_2)$ is an $F$-algebra with orthogonal involution of degree $6$ with trivial discriminant which has the property that $C(A_2,\sigma_2)\simeq B'\times B'^{\op}$.
    By \Cref{structure-clifford-awi}, we get that $A_2\sim B'^{\otimes2}\sim B^{\otimes 2}\sim Q$. 
    It follows that $(A_2,\sigma_2)\simeq\Ad_Q(h_2)$ for some non-degenerate skew-hermitian form $h_2$ over $(Q,\can_Q)$ with $\rk h_2=3$. 
    
    We consider now the $F$-algebra with orthogonal involution
    $$(A',\sigma')\,\,=\,\,\Ad_Q(h\perp h_2).$$
    Note that $\deg A'=16$ and $(A',\sigma')\in (A_0,\sigma_0)\boxplus(A_1,\sigma_1)\boxplus(A_2,\sigma_2)$.
    By \Cref{ort-sum-rel-e_1-e_2}, we obtain that $e_1(\sigma')$ is trivial and that
    $$e_2(\sigma')=e_2(\sigma_0)+e_2(\sigma_1)+e_2(\sigma_2)=[H]+[B]+[B']=0\quad\,\text{in}\, \Br(F)/\langle[Q]\rangle.$$
    As $u(F)\leq4$, we have that $\I^3F=0$. Hence, we conclude by \Cref{deg4m-ortogonal-isot-hyper-I^3F=0} that $(A',\sigma')$ is hyperbolic. 
    Therefore $h\perp h_2$ is hyperbolic. 
    Since $\rk h=5$ and $\rk(h\perp h_2)=8$, we conclude by \Cref{half-rk-subform-hyper->iso} that $h$ is isotropic. 
    This shows that $(A,\sigma)$ is isotropic.
\end{proof}

The following example shows that the upper bounds obtained in \Cref{bound-anis-alg-ind2-orth-I^3=0} and \Cref{orthogonal-hermit-u-quater-I^3=0} are optimal in the case where $u(F)=4$.

\begin{exmp}\label{optimality-u^+(Q)-for-u(F)=4}
    By \cite[Example 4.4]{BB23}, there exist a field $F$ with $\car F\neq 2$ and $u(F)=4$ and a central $F$-division algebra $D$ with $\exp D=4$ and $\deg D=8$. 
    Given such a field $F$, it follows by \Cref{deg10-ort-iso-equiv-cond} that there exists an anisotropic $F$-algebra  with orthogonal involution $(A,\sigma)$ of degree $10$ with trivial discriminant. 
    In view of \Cref{bound-anis-alg-ind2-orth-I^3=0}, this is the maximal possible degree for a central simple $F$-algebra carrying an anisotropic orthogonal involution when $u(F)=4$.
    
    The presence of an $F$-involution on $A$ implies that $\exp A\leq 2$. Since $u(F)=4$, it follows that $\ind A\leq 2$.
    As $\deg A >u(F)$ and $A$ carries an anisotropic orthogonal involution, $A$ cannot be split. 
    Hence $\ind A=2$, whereby $A$ is Brauer equivalent to an $F$-quaternion division algebra $Q$.
    We have $(A,\sigma)\simeq\Ad_Q(h)$ for a non-degenerate skew-hermitian form $h$ over $(Q,\can_Q)$. 
    Then $\rk h=\frac{1}2 \deg A=5$, and since $\sigma$ is anisotropic, it follows 
    that $h$ is anisotropic.
    In particular, $u^+(Q)\geq \rk h =5$. 
    Hence the upper bound $5$ for $u^+(Q)$ obtained by \Cref{orthogonal-hermit-u-quater-I^3=0} when $u(F)=4$ is attained in this example. 
\end{exmp}

\subsection*{Acknowledgements}
The authors are grateful to the referee for their careful reading, as well as for their valuable suggestions, comments and corrections.
This work was supported by the Fonds Wetenschappelijk Onderzoek – Vlaanderen in the FWO Odysseus Programme (project G0E6114N, \emph{Explicit Methods in Quadratic Form Theory}), by the FWO-Tourne\-sol programme (project VS05018N), by the Fondazione Cariverona in the programme Ricerca Scientifica di Eccellenza 2018 (project \emph{Reducing complexity in algebra, logic, combinatorics -- REDCOM}), by T\"{U}B\.{I}TAK-221N171, by the 2020 PRIN (project \emph{Derived and underived algebraic stacks and applications}) from MIUR, and by research funds from Scuola Normale Superiore.

\subsection*{Conflict of interest statement}
On behalf of all authors, the corresponding author states that there is no conflict of interest.

\subsection*{Data accessibility statement}
The  article describes entirely theoretical research.
All new data supporting the findings presented in this article are included in the manuscript.
Data sharing is not applicable to this article.

\end{document}